\documentclass[11pt,a4paper]{amsart}
\usepackage{amssymb}
\usepackage{amsmath}
\usepackage{todonotes}
\usepackage{esint}
\usepackage{mathtools}
\usepackage{changes}
\usepackage{tabu}
\usepackage[margin=1.5in]{geometry}
\usepackage[mathcal]{euscript}
\usepackage{color}
\definecolor{cobalt}{RGB}{61,89,171}
\usepackage[colorlinks,citecolor=cobalt,linkcolor=cobalt,urlcolor=cobalt,pdfpagemode=UseNone,backref = page]{hyperref}
\usepackage{makecell}

\newtheorem{thm}{Theorem}[section]
\newtheorem{corollary}[thm]{Corollary}
\newtheorem{lemma}[thm]{Lemma}
\newtheorem{claim}[thm]{Claim}
\newtheorem{proposition}[thm]{Proposition}

\theoremstyle{definition}
\newtheorem{definition}[thm]{Definition}
\newtheorem{example}[thm]{Example}

\theoremstyle{remark}
\newtheorem{remark}[thm]{Remark}

\newcommand{\hook}{\lrcorner \,}
\newcommand{\re}{\mathrm{Re}}
\newcommand{\im}{\mathrm{Im}}

\newcommand{\bs}{\backslash}
\newcommand{\la}{\langle}
\newcommand{\ra}{\rangle}

\newcommand{\bQ}{{\mathbb Q}}
\newcommand{\bR}{{\mathbb R}}
\newcommand{\bZ}{{\mathbb Z}}
\newcommand{\bC}{{\mathbb C}}
\newcommand{\bH}{{\mathbb H}}

\newcommand{\mfa}{\mathfrak{a}}

\newcommand{\mfg}{\mathfrak{g}}
\newcommand{\mfh}{\mathfrak{h}}
\newcommand{\mfn}{\mathfrak{n}}
\newcommand{\mfq}{\mathfrak{q}}

\newcommand{\mfz}{\mathfrak{z}}

\newcommand{\ad}{\operatorname{ad}}
\newcommand{\tr}{\operatorname{tr}}

\newcommand{\Alt}{\operatorname{Alt}}
\newcommand{\Der}{\operatorname{Der}}
\newcommand{\End}{\operatorname{End}}

\newcommand{\id}{\operatorname{id}}

\newcommand{\SO}{\operatorname{SO}}
\newcommand{\GL}{\operatorname{GL}}

\newcommand{\Sp}{\operatorname{Sp}}

\renewcommand{\Re}{\operatorname{Re}}
\renewcommand{\Im}{\operatorname{Im}}

\numberwithin{equation}{section}

\title{Complex symplectic structures on Lie algebras}


\author[G.~Bazzoni]{Giovanni Bazzoni}
\address[G.~Bazzoni]{Departamento de \'Algebra, Geometr\'ia y Topolog\'ia \\ Universidad Complutense de Madrid\\ Plaza de Ciencias 3\\ 28040 Madrid\\ Spain}
\email{gbazzoni@ucm.es}

\author[M.~Freibert]{Marco Freibert}
\address[M.~Freibert]{Department of Mathematics \\ King's College London \\ Strand \\ London WC2R 2LS \\ United Kingdom}
\email{marco.freibert@kcl.ac.uk}
\address{Mathematisches Seminar\\
Christian-Albrechts-Universit\"at zu Kiel\\
Ludewig-\linebreak Meyn-Strasse 4\\
D-24098 Kiel\\
Germany}
\email{freibert@math.uni-kiel.de}

\author[A.~Latorre]{Adela Latorre}
\address[A.~Latorre]{Centro Universitario de la Defensa\,-\,I.U.M.A., Academia General
Mili\-tar, Crta. de Huesca s/n. 50090 Zaragoza, Spain}
\email{adela@unizar.es}

\author[B.~Meinke]{Benedict Meinke}
\address[B.~Meinke]{Institut f\"ur Differentialgeometrie\\
Gottfried Wilhelm Leibniz Universit\"at Hannover\\
Welfengarten 1\\
D-30167 Hannover\\
Germany}
\email{meinke@math.uni-hannover.de}

\date{}

\begin{document}
\begin{abstract}
We investigate Lie algebras endowed with a complex symplectic structure and develop a method, called \emph{complex symplectic oxidation}, to construct certain complex symplectic Lie algebras of dimension $4n+4$ from those of dimension $4n$. We specialize this construction to the nilpotent case and apply complex symplectic oxidation to classify eight-dimensional nilpotent complex symplectic Lie algebras.
\end{abstract}

\maketitle
\section{Introduction}
A \emph{complex symplectic}, or \emph{holomorphic symplectic}, manifold $(M,J,\omega_\bC)$ is a smooth manifold $M$ together with an integrable almost complex structure $J$ and a closed, non-degenerate $(2,0)$-form $\omega_\bC$. Since $(M,J)$ is a complex manifold, the real dimension of $M$ is even; the fact that $\omega_\bC$ is non-degenerate has two further consequences: the \emph{complex} dimension of $M$ is equal to $2n$, and $\omega_\bC^n$ is a nowhere vanishing section of $\mathcal{K}_M$, the canonical bundle of $(M,J)$, which is therefore holomorphically trivial. Special examples of complex symplectic manifolds are given by hyperk\"ahler
 manifolds $(M,g,I,J,K)$, which carry naturally a complex symplectic structure $\omega_\bC\coloneqq\omega_J+i\omega_K$; here $\omega_J(\cdot,\cdot)\coloneqq g(\cdot,J\cdot)$ and similarly for $\omega_K$. The converse is not true: primary Kodaira surfaces are compact complex symplectic manifolds which are not hyperk\"ahler. In real dimensions $\geq 8$, Todorov \cite{Todorov} asked whether every irreducible compact complex symplectic manifold is k\"ahlerian. In \cite{Guan1}, Guan constructed a complex symplectic structure on the complexification of the Kodaira-Thurston manifold, which is 8-dimensional and does not admit a (hyper-)K\"ahler structure. Moreover, the same author constructed in \cite{Guan2,Guan3} a \emph{simply connected} non-K\"ahler compact complex symplectic manifold in each dimension $4n$; his examples are in spirit analogous to the two series of compact hyperk\"ahler manifolds obtained by Beauville in \cite{Beauville} and they use in a crucial way primary Kodaira surfaces. It is somehow surprising, and perhaps a reflection of the scarcity of compact hyperk\"ahler manifolds, that these are the only known simply connected examples.

It is well-known that primary Kodaira surfaces have the structure of \emph{nilmanifolds}; this is the quotient of a connected, simply connected, nilpotent Lie group~$G$ by a lattice~$\Gamma$ (i.e.~a discrete and cocompact subgroup; see Section \ref{examples} for the relevant details). Now note that a left-invariant complex symplectic structure on~$G$ induces a natural ``invariant'' complex symplectic structure on the nilmanifold~$\Gamma\bs G$ and that any K\"ahler nilmanifold is diffeomorphic to a torus \cite{Hasegawa}. Thus, complex symplectic nilmanifolds
never carry hyperk\"ahler metrics provided that~$G$ is non-abelian. Moreover, left-invariant complex symplectic structures on~$G$ are in one-to-one correspondence with complex symplectic structures on the associated Lie algebra~$\mfg$ and~$G$ admits a lattice if and only if~$\mfg$ admits a rational structure~\cite{Malcev}.

Summarizing, any complex symplectic structure on a non-abelian nilpotent Lie algebra admitting a rational structure provides us with an example of a compact (nil-)manifold admitting a complex symplectic structure which does not come from a hyperk\"ahler structure. This is the reason why we will focus in this paper on complex symplectic structures on nilpotent Lie algebras. Note also that examples of left-invariant complex symplectic structures on solvmanifolds have been considered in \cite{Cattaneo-Tomassini}.

This topic does not seem to have got much attention in the literature so far and only a very restrictive class of complex symplectic structures, called \emph{special Lagrangian}, on not necessarily nilpotent or solvable Lie algebras was studied in \cite{Cleyton-Poon-Ovando} in detail. 

The purpose of this paper is to delve into left-invariant complex symplectic structures on connected and simply connected Lie groups. When the group is nilpotent or solvable, one may obtain compact examples by modding out a suitable lattice.

The paper is organized as follows:
\begin{itemize}
 \item in Section \ref{geo_struct} we show how complex symplectic structures are related to many other geometric structures and state some open questions;
 \item Section \ref{basics} contains some generalities on complex symplectic structures on Lie algebras;
 \item in Section \ref{oxidation-reduction} we present a method to construct complex symplectic structures on Lie algebras of dimension $4n+4$ starting from a $4n$-dimensional complex symplectic Lie algebra. We call this construction \emph{complex symplectic oxidation} in analogy with the symplectic oxidation, see \cite{Baues-Cortes};
\item in Section \ref{explicit-8-d} we concentrate on the nilpotent case. We show that every complex symplectic nilpotent Lie algebra (NLA) whose underlying complex structure is quasi-nilpotent can be obtained by symplectic
oxidation. Moreover, we show that 8-dimensional complex symplectic NLAs do not admit complex structures of strongly non-nilpotent type. As a consequence, we conclude that any 8-dimensional complex symplectic NLA is the complex symplectic oxidation of a 4-dimensional one. We then explicitly compute the oxidation data for the 4-dimensional complex symplectic NLAs, thus classifying all 8-dimensional complex symplectic NLAs;
 \item Section \ref{examples} contains various examples of nilmanifolds and solvmanifolds endowed with both complex symplectic and related geometric structures.
\end{itemize}

\section{Geometric structures on complex manifolds}\label{geo_struct}

We next explain how, apart from hyperk\"ahler, complex symplectic manifolds in general and complex symplectic Lie algebras in particular are related to a number of geometric structures and give rise to interesting problems. 

In the language of principal bundles, the existence of a geometric structure on an $n$-dimensional manifold can be rephrased in terms of the existence of a reduction of the structure group of the frame bundle of a manifold, which is a principal $\GL(n,\bR)$-bundle, to a certain Lie subgroup $G\subset \GL(n,\bR)$. In this case, one says that the manifold is endowed with a $G$-structure. For every $G$-structure there is a notion of integrability, which is equivalent to the existence of an atlas, which is $G$-adapted up to order two terms.

As an example, a Riemannian metric is an integrable $\mathrm{O}(n)$-structure and a corresponding atlas is given by normal coordinates. Similarly, if $n=2m$,
a symplectic structure is an integrable $\mathrm{Sp}(2m,\bR)$-structure, a corresponding atlas being given by Darboux coordinates, and
a complex structure is an integrable $\GL(m,\bC)$-structure, where a corresponding atlas is simply a holomorphic one.

It is well-known that these three structures can appear simultaneously on a manifold, and that the following relations among subgroups of $\GL(2m,\bR)$ hold:
\begin{equation}\label{symp-real}
\mathrm{O}(2m)\cap \Sp(2m,\bR)\!=\!\mathrm{O}(2m)\cap \GL(m,\bC)\!=\!\Sp(2m,\bR)\cap \GL(m,\bC)\!=\!\mathrm{U}(m)\,.
\end{equation}
This fact is sometimes referred to as K\"ahler geometry, the geometry of an integrable $\mathrm{U}(m)$-structure, lying at the intersection of Riemannian, symplectic and complex geometry. It should also be noted that $\mathrm{U}(m)$ is one of the special holonomy groups of Riemannian manifolds appearing in the celebrated Berger list, see \cite{Berger}.
\begin{remark}\label{compatible}
Any manifold carries a Riemannian metric. There exist examples of symplectic manifolds without K\"ahler structures \cite{McDuff,Oprea-Tralle}, of complex manifolds without K\"ahler structures \cite{Calabi-Eckmann,Hopf} and even of manifolds which are complex and symplectic but still not K\"ahler, see \cite{Bazzoni-Fernandez-Munoz,Bazzoni-Munoz,Thurston}.

\end{remark}

Things can be complicated a bit by fixing an integrable almost complex structure~$I$ on a manifold~$M$ of \emph{even} complex dimension $2n$; with respect to this choice, the frame bundle of $M$ has naturally the structure of a $\GL(2n,\bC)$-principal bundle, whose associated $\GL(2n,\bC)$-structure is integrable. Consider now the following subgroups of $\GL(2n,\bC)$:
\begin{align*}
\GL(n,\bH) & =\{A\in \GL(2n,\bC) \mid A\mathcal{J}=\mathcal{J}\bar{A}\}\,,\\
\Sp(2n,\bC) &=\{A\in \GL(2n,\bC) \mid A^t\mathcal{J} A=\mathcal{J}\}\,,\\
\mathrm{U}(2n) & =\{A\in \GL(2n,\bC) \mid \bar{A}^tA=\mathrm{Id}\}\,,
\end{align*}
where 
\[
\mathcal{J}=
\begin{pmatrix}
 0 & -\mathrm{Id}\\
 \mathrm{Id} & 0
\end{pmatrix}\,.
\]
$\GL(n,\bH)$ consists of those linear transformations of $\bC^{2n}$ which are $\bH$-linear ($\bH$ being the quaternions), where the identification $\bH^n\cong\bC^{2n}$ is given by $z+w\cdot j\mapsto (z,\bar{w})$. $\Sp(2n,\bC)$ denotes those linear transformations of $\bC^{2n}$ which preserve the symplectic form
\[
\omega_\bC=\sum_{i=1}^ne^i\wedge e^{n+i}\,,
\]
being $\{e_i\}_{i=1}^{2n}$ the canonical basis of $\bC^{2n}$. In this setting, the statement which mirrors \eqref{symp-real} is
\[
\mathrm{U}(2n)\cap \Sp(2n,\bC)=\mathrm{U}(2n)\cap \GL(n,\bH)=\Sp(2n,\bC)\cap \GL(n,\bH)=\Sp(n)\,,
\]
where $\Sp(n)$ is the group of linear transformations of $\bH^n$ which preserve the quaternionic scalar product $\langle q,r\rangle=\bar{q}^t r$.\\

In the language of $G$-structures, we have the following nomenclature:
\begin{itemize}
\item an integrable $\GL(n,\bH)$-structure is a \emph{hypercomplex} structure. By spelling out this condition, one finds that $(M,I)$ has a second integrable almost complex structure $J$ which anticommutes with $I$. Thus $K\coloneqq I\circ J$ is a third integrable almost complex structure. Each tangent space to a hypercomplex manifold inherits the structure of a quaternionic vector space. Further details on hypercomplex structures can be found in \cite{Joyce,Obata}; in particular, 8-dimensional nilpotent Lie algebras (NLAs) which carry a hypercomplex structure have been classified in \cite{Dotti-Fino};
\item an integrable $\Sp(2n,\bC)$-structure is a \emph{complex symplectic} or \emph{holomorphic symplectic} structure, see \cite{Boalch}; this is the topic of this paper. In the recent papers \cite{Anthes-Cattaneo-Rollenske-Tomassini,Cattaneo-Tomassini}, the authors show that, in the compact case, complex symplectic manifolds which satisfy the $\partial\overline\partial$-lemma retain some properties of hyperk\"ahler manifolds. Since a complex manifold which satisfies the $\partial\overline\partial$-lemma is formal, see \cite{Deligne-Griffiths-Morgan-Sullivan}, and by a result of Hasegawa \cite{Hasegawa}, a formal nilmanifold is diffeomorphic to a torus, except for tori, the results of \cite{Anthes-Cattaneo-Rollenske-Tomassini,Cattaneo-Tomassini} do not apply to complex symplectic nilmanifolds. However, we shall show that our construction allows to recover an example described in \cite{Cattaneo-Tomassini}, which is a certain 8-dimensional solvmanifold;
\item an integrable $\Sp(n)$-structure is a \emph{hyperk\"ahler} structure. The metric of a hyperk\"ahler manifold is Ricci-flat and such manifolds have attracted a great deal of interest, both from the mathematical and the physical point of view. We refer the reader to \cite{Huybrechts,Joyce2,Salamon2} and the references therein.
\end{itemize}

Thus one can say that hyperk\"ahler geometry lies at the intersection of K\"ahler, hypercomplex and complex symplectic geometry. Again we should note that $\Sp(n)$ is one of the special holonomy groups of the Berger list. Concerning the interplay of these structures on compact manifolds, it was shown by Beauville in \cite[Proposition 4]{Beauville} that on a compact complex manifold of K\"ahler type, of complex dimension $2n$, the existence of a K\"ahler metric whose holonomy group is contained in $\Sp(n)$ is equivalent to the existence of a complex symplectic structure. Hence, a compact K\"ahler manifold which admits a complex symplectic structure is necessarily hyperk\"ahler. As for manifolds which are hypercomplex and K\"ahler, Verbitsky proved in \cite{Verbitsky} that if a compact hypercomplex manifold $(M,I,J,K)$ is such that $(M,I)$ admits a K\"ahler metric, then $(M,I)$ is complex symplectic and hyperk\"ahler a fortiori, thanks to Beauville's result\footnote{As explained in \cite{Verbitsky}, it is however not a priori clear that the hyperk\"ahler structure provided by Beauville's argument has $(I,J,K)$ as underlying hypercomplex structure, although Verbitsky conjectures this should be the case.}. The situation here is way more constrained than in the real case, see Remark \ref{compatible}. In particular, one can only hope to find examples of compact manifolds which are complex symplectic and hypercomplex but not hyperk\"ahler. Example \ref{ex:csandhc} below shows that such compact manifolds exist.\\

Complex symplectic structures are also related to \emph{hypersymplectic structures}, introduced by Hitchin in \cite{Hitchin}. As it 
happens for every complex simple Lie group, $\Sp(2n,\bC)$ has the following two real forms: the \emph{compact} one, which is $\Sp(n)$, 
and the \emph{split} one, which is $\Sp(2n,\bR)$. A hypersymplectic structure on a complex symplectic manifold $(M,J,\omega_\bC)$ of 
real dimension $4n$ is a further reduction of the frame bundle from $\Sp(2n,\bC)$ to $\Sp(2n,\bR)$. This is accomplished by a real structure, i.e.~an endomorphism $S\colon TM\to TM$ which satisfies 
$S^2=\mathrm{Id}$ and anticommutes with $J$. $\Sp(2n,\bR)$ is not compact, but it acts on $\bR^{4n}$ as a subgroup of $\SO(2n,2n)$; hypersymplectic metrics live therefore outside the Riemannian world, 
in the realm of neutral metrics, i.e.~metrics with signature $(2n,2n)$. Further general information on hypersymplectic structures can be 
found in \cite{Dancer-Swann}; hypersymplectic structures on NLAs and nilmanifolds
have been investigated in \cite{Andrada,Andrada-Dotti,Fino-Pedersen-Poon-Sorensen}.

%

\section{Basic definitions and properties}\label{basics}

\subsection{Generalities on Lie algebras}

Let $G$ be a Lie group with Lie algebra $\mfg$. To a $\mfg$-module $W$ we can associate the Chevalley-Eilenberg complex $(C^\bullet(\mfg;W),d_W)$, whose elements are linear maps $\alpha\colon \Lambda^q\mfg\to W$; the differential $d_W\colon C^q(\mfg;W)\to C^{q+1}(\mfg;W)$ is defined by
\begin{align}\label{CE_diff}
 (d_W\alpha)(X_0,\ldots,X_q)&=\sum_{i<j}(-1)^{i+j}\alpha([X_i,X_j],X_0,\ldots,\hat{X_i},\ldots,\hat{X_j},\ldots,X_q)\nonumber\\
 &+\sum_{i=0}^qX_i\cdot\alpha(X_0,\ldots,\hat{X_i},\ldots,X_q)\,.
\end{align}

For a trivial $\mfg$-module $W$ we denote $d_W$ by $d$ or $d_\mfg$, if confusion arises; in this case, \eqref{CE_diff} reduces to 
\begin{equation*}
 (d\alpha)(X_0,\ldots,X_q)=\sum_{i<j}(-1)^{i+j}\alpha([X_i,X_j],X_0,\ldots,\hat{X_i},\ldots,\hat{X_j},\ldots,X_q)\,.
\end{equation*}

When $W$ equals the trivial module $R$, $(C^\bullet(\mfg;\bR),d)$ is the usual Chevalley-Eilenberg complex $(\Lambda^\bullet\mfg^*,d)$, which computes the Lie agebra cohomology of $\mfg$. In this case $d\colon\mfg^*\to\Lambda^2\mfg^*$ satisfies
\[
 d\alpha(X,Y)=-\alpha([X,Y])\,,
\]
hence $d$ is the transpose of the linear map $[\cdot,\cdot]\colon\Lambda^2\mfg\to\mfg$ and $d^2=0$ is equivalent to the Jacobi identity in $\mfg$; thus, the knowledge of $d$ is equivalent to the knowledge of $[\cdot,\cdot]$. To denote a Lie algebra $\mfg$ we use Salamon's notation, which is best explained by an example: $\mfg=(0,0,0,12,13+3\cdot 14)$ means that the 5-dimensional Lie algebra $\mfg$ has basis $\{e_1,\ldots,e_5\}$ with dual basis $\{e^1,\ldots,e^5\}$ and differential $d\colon\mfg^*\to\Lambda^2\mfg^*$ given by
\[
 de^1=0, \quad  de^2=0, \quad  de^3=0 \quad de^4=e^{12} \quad \textrm{and} \quad de^5=e^{13}+3e^{14}\,;
\]
the short-hand notation $e^{ij}\coloneqq e^i\wedge e^j$ will be used throughout.\\

An \emph{almost complex structure} on a Lie algebra $\mfg$ is an endomorphism $J\colon\mfg\to\mfg$ satisfying $J^2=-\id$. The natural extension of $J$ to the complexification $\mfg_\bC=\mfg\otimes\bC$ diagonalizes with eigenvalues $i$ and $-i$; the eigenspaces are
\[
 \mfg^{(1,0)}=\{X-iJX \mid X\in\mfg\} \quad \textrm{and} \quad  \mfg^{(0,1)}=\{X+iJX \mid X\in\mfg\}\,.
\]
$J$ also gives an endomorphism $J^*$ of $\mfg^*$, $J^*\alpha=\alpha\circ J$, providing in turn a splitting of $\mfg^*_\bC=\mfg^*\otimes\bC$
\[
 \mfg^{*(1,0)}=\{\alpha-iJ^*\alpha \mid \alpha\in\mfg^*\} \quad \textrm{and} \quad  \mfg^{*(0,1)}=\{\alpha+iJ^*\alpha \mid \alpha\in\mfg^*\}\,.
\]
With this notation, giving an almost complex structure on $\mfg$ is equivalent to giving the subspace $\mfg^{*(1,0)}$. By definition, $J$ is integrable if and only if the Nijenhuis tensor
\[
 N_J(X,Y)\coloneqq [X,Y]+J([JX,Y]+[X,JY])-[JX,JY]
\]
vanishes for every $X,Y\in\mfg$. This is known to be equivalent to the subspace $\mfg^{(1,0)}$ being closed with respect to the Lie bracket and also to
\[
d(\mfg^{*(1,0)})\subset\Lambda^2\mfg^{*(1,0)}\oplus \mfg^{*(1,0)}\otimes\mfg^{*(0,1)}\,.
\]
We call an integrable almost complex structure simply a \emph{complex structure}.

Finally, a \emph{symplectic structure} on a Lie algebra $\mfg$ of dimension $2n$ is nothing but a cocycle $\omega\in\Lambda^2\mfg^*$ which is non-degenerate, i.e.~$\omega^n\neq 0$.


%
%
%

\subsection{Complex symplectic manifolds and Lie algebras}
We start with the usual definition of a \emph{complex symplectic structure}:
\begin{definition}
A \emph{complex symplectic structure} on a manifold $M$ is a pair $(J,\omega_{\bC})$ consisting of a complex structure $J$
and a 2-form $\omega_\bC\in\Omega^2(M,\bC)$ which is non-degenerate, closed and of type $(2,0)$. In this case, we call $(M,J,\omega_{\bC})$ a \emph{complex symplectic manifold}.
\end{definition}
As already stated in the introduction, a complex symplectic structure is a special kind of an $\mathrm{Sp}(2n,\bC)$-structure.
Since $\mathrm{Sp}(2n,\bC)=\GL(2n,\bC)\cap \mathrm{Sp}(4n,\bR)$, there has to be an equivalent description of
a complex symplectic structure in terms of a pair $(J,\omega)$ consisting of an almost complex
structure $J$ and non-degenerate two-form $\omega$ with certain properties. This description is as follows:
\begin{lemma}\label{le:equivalencecomplexsymplectic}
Let $M$ be a manifold. Then the set of all complex symplectic structures $(J,\omega_\bC)$ on $M$ is bijective to the set of all pairs $(J,\omega)$ of complex structures $J$ and symplectic 2-forms $\omega$ such that $J$ is symmetric with respect to $\omega$,
i.e.~$\omega(JX,Y)=\omega(X,JY)$ for all $X,Y\in \mathfrak{X}(M)$ holds. The bijection is given by $(J,\omega_\bC)\mapsto (J,\Re(\omega_\bC))$
with inverse $(J,\omega)\mapsto (J,\omega-i \omega(J\cdot,\cdot))$.
\end{lemma}
\begin{proof}
Let $(J,\omega_\bC)$ be a complex symplectic structure and $X,Y\in \mathfrak{X}(M)$. Since $\omega_\bC$ is of type $(2,0)$, one has $\omega_\bC(J X, JY)=-\omega_\bC(X,Y)$, giving $\Re(\omega_\bC)(JX,Y)=-\Re(\omega_{\bC})(JX,J(JY))=\Re(\omega_{\bC})(X,JY)$. Thus $\Re(\omega_{\bC})$
is, in fact, symmetric with respect to $J$. Moreover,
\begin{equation*}
\begin{split}
\Re(\omega_{\bC})(JX,Y)+i \Im(\omega_{\bC})(JX,Y)&=\omega_{\bC}(J X,Y)=i\omega_{\bC}(X,Y)\\
&=i\Re(\omega_{\bC})(X,Y)-\Im(\omega_{\bC})(X,Y),
\end{split}
\end{equation*}
i.e.~$\Im(\omega_{\bC})=-\Re(\omega_{\bC})(J\cdot,\cdot)$.

To finish the proof, we need to show that $\omega(J\cdot,\cdot)$ is closed for any pair $(J,\omega)$ of a complex structure $J$ and a symplectic form $\omega$ such that $J$ is symmetric with respect to $\omega$. However, a short computation shows that
\begin{equation*}
d(\omega(J\cdot,\cdot))(X,Y,Z)=\frac{1}{2}\left(d\omega(JX,Y,Z)+d\omega(X,JY,Z)+d\omega(X,Y,JZ)\right)=0
\end{equation*}
for any $X,Y,Z\in \mathfrak{X}(M)$.
\end{proof}

Thanks to Lemma \ref{le:equivalencecomplexsymplectic}, we will from now on also call a pair $(J,\omega)$ on a manifold $M$
consisting of an integrable almost complex structure $J$ and a symplectic form $\omega$ such that $\omega$ is symmetric with respect to $J$ a \emph{complex symplectic structure} and $(M,J,\omega)$ a \emph{complex symplectic manifold}.

We are here mainly interested in left-invariant complex symplectic structures on a Lie group $G$. As left-invariant structures are fully 
determined by their value in the identity $e$ of $G$, i.e.~on the associated Lie algebra $\mfg=T_e G$, we may describe left-invariant complex symplectic structures on $G$ equivalently as \emph{complex symplectic structures} on $\mfg$. Then we say that $(\mfg,J,\omega)$, or $(\mfg,J,\omega_{\bC})$, is a \emph{complex symplectic Lie algebra}. Note that then $J$ is a complex structure on $\mfg$
and $\omega$ is a symplectic structure on $\mfg$ such that $J$ is symmetric with respect to $\omega$.

\begin{remark}\label{re:adaptedbasis}
Since a complex symplectic structure $(J,\omega)$ on a $4n$-dimensional manifold $M$ is an $\mathrm{Sp}(2n,\bC)$-structure, each point $p\in M$ admits an
$\mathrm{Sp}(2n,\bC)$-orbit of bases $\{e_1,\ldots,e_{4n}\}$ of $T_p M$ such that $J_p$ and $\omega_p$ are given by
\begin{equation*}
J_p\coloneqq\sum_{i=1}^{2n} e^{2i}\otimes e_{2i-1}-e^{2i-1}\otimes e_{2i},\quad \omega_p\coloneqq\sum_{j=1}^n e^{4j-3}\wedge e^{4j}+e^{4j-2}\wedge e^{4j-1}.
\end{equation*}
Such a basis is called an \emph{adapted basis} for $(J,\omega)$ at $p\in M$. Clearly, for a left-invariant complex symplectic structure $(J,\omega)$ on a $4n$-dimensional Lie group $G$, a global left-invariant frame $\{e_1,\ldots,e_{4n}\}$, i.e.~a basis of the associated Lie algebra $\mfg$, may give adapted bases at any point in $G$, and we then also call $\{e_1,\ldots,e_{4n}\}$ an \emph{adapted basis} for $(\mfg,J,\omega)$.

\end{remark}

\section{Complex symplectic reduction and oxidation}\label{oxidation-reduction}
Let $(\mfg,\omega)$ be a symplectic Lie algebra. An ideal $\mfa$ is 
$\omega$-isotropic if $\omega(X,Y)=0$ for every $X,Y\in\mfa$; in this case, $\mfa^{\perp}$ is a Lie subalgebra containing $\mfa$ and 
$\omega$ induces a symplectic form on the Lie algebra $\mfa^{\perp}/\mfa$, where here and in the following the orthogonal complement is 
taken with respect to $\omega$. This process is called \emph{symplectic reduction}, see e.g.~\cite{Baues-Cortes}. 

We can easily extend this construction to complex symplectic Lie algebras as follows:

\begin{lemma}[Complex symplectic reduction]\label{le:complexsymplecticreduction}
Let $(\mfg,J,\omega)$ be a complex symplectic Lie algebra and let $\mfa$ be an $\omega$-isotropic $J$-invariant ideal. Then $(J,\omega)$ induces 
a complex symplectic structure $(\bar J,\bar\omega)$ on $\bar{\mfg}:=\mfa^{\perp}/\mfa$ and $\mfa^{\perp}$
is a $J$-invariant subalgebra of $\mfg$. If $\mfa$ is central, then $\mfa^{\perp}$ contains the commutator ideal $[\mfg,\mfg]$
and so is, in particular, an ideal in $\mfg$.
\end{lemma}
\begin{proof}
We have an induced symplectic structure $\bar\omega$ on $\bar{\mfg}$ given by $\bar{\omega}(X+\mfa,Y+\mfa)\coloneqq\omega(X,Y)$ for all $X,Y\in\mfa^{\perp}$
by symplectic reduction and one easily checks that if $\mfa$ is central, then $\mfa^{\perp}$ is actually an ideal which contains the commutator ideal $[\mfg,\mfg]$, see also \cite{Baues-Cortes}.

Now the symmetry of $J$ with respect to $\omega$ and the $J$-invariance of $\mfa$ imply that $\mfa^{\perp}$ is $J$-invariant.
Thus, $J$ induces an integrable almost complex structure $\bar{J}$ on $\mfa^\perp/\mfa$ by $\bar{J}(X+\mfa)\coloneqq J(X)+\mfa$ for $X\in \mfa^{\perp}$, which is
symmetric with respect to $\bar{\omega}$.
\end{proof}

When $\mfa$ is a $J$-invariant \emph{central} $2$-dimensional ideal, there is an inverse construction to complex symplectic reduction. We call it \emph{complex symplectic oxidation}. It is analogous to the so-called \emph{symplectic oxidation} introduced in \cite{Baues-Cortes} as a converse to symplectic reduction with respect to a central one-dimensional ideal. First, a trivial observation:
\begin{claim}\label{cl:centralJinvariant2Disotropic}
Any $J$-invariant $2$-dimensional ideal $\mfa$ is $\omega$-isotropic.
\end{claim}
\begin{proof}
 Fix $X\in \mfa$ nonzero; then $\{X,JX\}$ is a basis of $\mfa$ and we have
 \[
  \omega(X,JX)=\omega(JX,X)=-\omega(X,JX) \ \Rightarrow \ \omega(X,JX)=0\,.
 \]
\end{proof}

We next describe in detail complex symplectic oxidation. Let $(\bar\mfg,\bar J,\bar\omega)$ be a complex symplectic Lie algebra  obtained by complex symplectic reduction from a complex symplectic Lie algebra $(\mfg,J,\omega)$ of dimension $4n+4$ using a central two-dimensional $J$-invariant ideal $\mfa$. Identifying $\bar\mfg$ with a $J$-invariant subspace of $\mfa^{\perp}$ complementary to $\mfa$, we have
%
%
%
\[
\bar\omega=\omega|_{\bar\mfg\times\bar\mfg} \quad \textrm{and} \quad \bar J=J|_{\bar\mfg}\,. 
\]

If we also choose a $J$-invariant complement $V$ of $\mfa^{\perp}$ in $\mfg$ which is $\omega$-orthogonal to $\bar{\mfg}$, then the map
\[
\mfa\to V^*, \quad w\mapsto \omega(w,\cdot)|_V
\]
is a $J$-equivariant isomorphism and we may use it to identify $\mfa$ with $V^*$. So to invert complex symplectic reduction with respect to a 2-dimensional central $J$-invariant ideal we need to start with a $4n$-dimensional complex symplectic Lie algebra $(\bar\mfg,\bar J,\bar\omega)$ and a 2-dimensional real vector space $V$ endowed with an almost complex structure $I$ and define the vector space $\mfg:=V\oplus \bar\mfg\oplus V^*$. Under the above identification,
\begin{equation}\label{eq:omegaong}
\omega((v,X,\alpha),(w,Y,\gamma))=\bar{\omega}(X,Y)+\alpha(w)-\gamma(v)
\end{equation}
and
\begin{equation}\label{eq:Jong}
  J(v,X,\alpha):=(Iv,\bar JX,I^*\alpha)
\end{equation}
for all $(v,X,\alpha),(w,Y,\gamma)\in\mfg$. Note that then $J$ is symmetric with respect to $\omega$.

It remains to define a Lie bracket on $\mfg=V\oplus\bar{\mfg}\oplus V^*$. Thereto, note that $V^*=\mfa$ has to be central and so $[\mfg,\mfg]$ has to be in $\bar{\mfg}\oplus V^*=\mfa^{\perp}$
by Lemma \ref{le:complexsymplecticreduction}. Hence, the non-zero (up to skew-symmetry) parts of the Lie bracket $[\cdot,\cdot]$ on $\mfg$ are given by
\begin{eqnarray}
 \left[(v,0,0),(w,0,0)\right] &=&(0,\nu(v,w),\tau(v,w)), \label{eq:Liebracketong1}\\
 \left[(v,0,0),(0,X,0)\right] &=&(0,f(v,X),g(v,X)), \label{eq:Liebracketong2}\\
 \left[(0,X,0),(0,Y,0)\right] &=&(0,[X,Y]_{\bar{\mfg}},\beta(X,Y)), \label{eq:Liebracketong3}
\end{eqnarray}
for $v,w\in V$ and $X,Y\in \bar{\mfg}$, where 
\begin{itemize}
 \item $\nu\colon\Lambda^2 V\to \bar{\mfg}$ $(\Leftrightarrow \nu\in\Lambda^2 V^*\otimes\bar\mfg)$,
 \item $\tau\colon\Lambda^2 V\to V^*$ $(\Leftrightarrow \tau\in\Lambda^2 V^*\otimes V^*)$,
 \item $f\colon V\otimes \bar\mfg\to\bar{\mfg}$ $(\Leftrightarrow f\in V^*\otimes\bar\mfg^*\otimes\bar\mfg\cong V^*\otimes\mathrm{End}(\bar\mfg))$,
 \item $g\colon V\otimes \bar{\mfg}\to V^*$ $(\Leftrightarrow g\in V^*\otimes\bar\mfg^*\otimes V^*)$,
 \item $\beta\colon \Lambda^2 \bar{\mfg}\to V^*$ $(\Leftrightarrow \beta\in \Lambda^2\bar\mfg^*\otimes V^*)$,
\end{itemize}
and $[\cdot,\cdot]_{\bar{\mfg}}$ is the Lie bracket of $\bar{\mfg}$. Altogether, the Lie bracket on $\mfg$ is
\begin{align*}
 \left[(v,X,\alpha),(w,Y,\gamma)\right]=&(0,\nu(v,w)+f(v,Y)-f(w,X)+[X,Y]_{\bar{\mfg}},\\
 &g(v,Y)-g(w,X)+\tau(v,w)+\beta(X,Y))\,.
\end{align*}
We need to reformulate the following in terms of conditions on $\nu$, $\tau$, $f$, $g$ and $\beta$:
\begin{description}
 \item[1] $[\cdot,\cdot]$ satisfies the Jacobi identity,
 \item[2] $\omega$ is closed and
 \item[3] $J$ is integrable
\end{description}
This will be done in the next three lemmas.\\


We introduce notation needed in the first lemma:
\begin{itemize}
 \item For $\xi,\eta\in V^*\otimes \mathrm{End}(\bar\mfg)$, we denote by $\{\xi,\eta\}\in \Lambda^2 V^*\otimes\mathrm{End}(\bar{\mfg})$ their bracket as endomorphisms of $\bar\mfg$, namely
\[
 \{\xi,\eta\}(v,w)=\xi(v)\circ\eta(w)-\xi(w)\circ\eta(v)\,.
\]
\item Given $\rho\in\Lambda^2\bar\mfg^*\otimes V^*$ and $D\in\mathrm{Der}(\bar{\mfg})$, define
$D.\rho\in\Lambda^2\bar\mfg^*\otimes V^*$ by
\[
 (D.\rho)(X,Y)=\rho(DX,Y)+\rho(X,DY)\,.
\]
\item Given $f$ and $g$ as above, we define $g\circ f\colon V\otimes V\otimes\bar\mfg\to V^*$ by
\begin{equation*}
 (g\circ f)(v,w,X)=g(v,f(w,X))\,.
\end{equation*}
\item Let $U$, $W$ and $Z$ be vector spaces; given a tensor $T\colon U\otimes U\otimes W\to Z$, we denote by $\Alt(T)\colon\Lambda^2 U\otimes W\to Z$ its skew-symmetrization in the first two variables, i.e.
\[
 \Alt(T)(u,v,w)=\frac{1}{2}\left(T(u,v,w)-T(v,u,w)\right)\,.
 \]
\end{itemize}

\begin{lemma}\label{lemma-jacobi}
The anti-symmetric product $[\cdot,\cdot]$ defined by \eqref{eq:Liebracketong1}, \eqref{eq:Liebracketong2} and \eqref{eq:Liebracketong3} satisfies the Jacobi identity if and only if the following conditions hold:
\begin{enumerate}
\item[(i)] $f\in V^*\otimes \mathrm{Der}(\bar{\mfg})$,
\item[(ii)] $\mathrm{ad}_\nu=\{f,f\}$,
\item[(iii)] $d_{\bar\mfg}g=-f.\beta$,
\item[(iv)] $\nu\hook\beta=2\Alt(g\circ f)$,
\item[(v)] $d_{\bar\mfg}\beta=0$.
\end{enumerate}
\end{lemma}
\begin{proof}
 Since $V^*$ is central in $\mfg$ and $V$ is two-dimensional, it is enough to check the Jacobi identity in three situations:
\begin{description}
 \item[A] $(v,0,0)$, $(w,0,0)$, $(0,X,0)$,
 \item[B] $(v,0,0)$, $(0,X,0)$, $(0,Y,0)$ and
 \item[C] $(0,X,0)$, $(0,Y,0)$, $(0,Z,0)$.
\end{description}
For \textbf{A}, writing down the Jacobi identity in each summand one gets
\begin{eqnarray}
 [\nu(v,w),X]_{\bar{\mfg}} &=& f(v,f(w,X))-f(w,f(v,X))\label{eq:condition1}\\
 \beta(\nu(v,w),X) &=&g(v,f(w,X))-g(w,f(v,X))\label{eq:condition2}\,;
\end{eqnarray}
doing the same for \textbf{B} one arrives to
\begin{eqnarray}
 f(v,[X,Y]_{\bar{\mfg}}) &=& [f(v,X),Y]_{\bar{\mfg}}+[X,f(v,Y)]_{\bar{\mfg}}\label{eq:condition3}\\
 g(v,[X,Y]_{\bar{\mfg}}) &=& \beta(f(v,X),Y)+\beta(X,f(v,Y))\label{eq:condition4}\,.
\end{eqnarray}
Using the notation above, condition \eqref{eq:condition1} is simply 
\begin{equation*}
\mathrm{ad}_\nu=\{f,f\} 
\end{equation*}
in $\Lambda^2 V^*\otimes\mathrm{End}(\bar{\mfg})$, which is part~\textrm{(ii)} of the statement above, while~\eqref{eq:condition3} gives
part~\textrm{(i)}, namely, 
\begin{equation*}
f\in V^*\otimes \mathrm{Der}(\bar{\mfg})\,.
\end{equation*}

Moreover, \eqref{eq:condition4} reads 
\begin{equation*}
 d_{\bar\mfg}g=-f.\beta
\end{equation*}
in $V^*\otimes \Lambda^2\bar\mfg^*\otimes V^*$, i.e.~ $d_{\bar\mfg}g(v)=-f(v).\beta$ for every $v\in V$. This is part \textrm{(iii)} of the lemma.

Using our notation, \eqref{eq:condition2} becomes 
 \begin{equation*}
 \nu\hook\beta=2\Alt(g\circ f)
\end{equation*}
in $\Lambda^2 V^*\otimes\bar\mfg^*\otimes V^*$, which coincides with \textrm{(iv)} in the statement.

Finally, we deal with \textbf{C} above; \eqref{eq:Liebracketong3} shows that the Jacobi identity on vectors $(0,X,0)$, $(0,Y,0)$ and $(0,Z,0)$ is equivalent to
\begin{eqnarray}
 [[X,Y]_{\bar{\mfg}},Z]_{\bar{\mfg}}+[[Y,Z]_{\bar{\mfg}},X]_{\bar{\mfg}}+[[Z,X]_{\bar{\mfg}},Y]_{\bar{\mfg}} &=& 0\nonumber\\
 \beta([X,Y]_{\bar{\mfg}},Z)+\beta([Y,Z]_{\bar{\mfg}},X)+\beta([Z,X]_{\bar{\mfg}},Y) &=&0\label{eq:condition4b}\,;
\end{eqnarray}
The first equation is just the Jacobi identity on $\bar\mfg$. For the second one, note that $\beta\in\Lambda^2 \bar\mfg^*\otimes V^*$ and \eqref{eq:condition4b} is simply $d_{\bar\mfg}\beta=0$. This gives part \textrm{(v)} of the lemma and concludes the proof.
\end{proof}
 
 
 For the second lemma, we need the following convention: since $V^*\otimes\bar\mfg^*\otimes V^*$ and $V^*\otimes V^*\otimes\bar\mfg^*$ are canonically isomorphic we can think of $g$ as an element of the latter vector space;  as such, we can decompose it in its symmetric and skew-symmetric part, $g=S+A\in S^2V^*\otimes \bar\mfg^*\oplus\Lambda^2 V^*\otimes \bar\mfg^*$; according to the notation introduced before Lemma \ref{lemma-jacobi}, $A=\Alt(g)$.
 
 
\begin{lemma}\label{lemma-omega}
Assume that $[\cdot,\cdot]$ defined by \eqref{eq:Liebracketong1}, \eqref{eq:Liebracketong2} and \eqref{eq:Liebracketong3} is a Lie bracket on $\mfg$.
Then the $2$-form $\omega$ defined by \eqref{eq:omegaong} is closed if and only if
$$\nu\hook\bar\omega=2A \qquad \text{and} \qquad \beta=-f.\bar\omega.$$
\end{lemma}

\begin{proof}
 First of all notice that $V^*\hook d\omega=0$ since $V^*$ is central and
$[\mfg,\mfg]\subseteq \bar{\mfg}\oplus V^*=(V^*)^{\perp}$. Moreover, one has $d\omega|_{\Lambda^3 \bar{\mfg}}=d_{\bar{\mfg}} \bar{\omega}=0$. Hence, we have to deal with two cases:
\begin{description}
 \item[D] $(v,0,0)$, $(w,0,0)$, $(0,X,0)$ and
 \item[E] $(v,0,0)$, $(0,X,0)$, $(0,Y,0)$.
\end{description}
For \textbf{D}, applying \eqref{eq:omegaong}, one gets
\begin{equation}\label{eq:condition5}
 \bar\omega(\nu(v,w),X)=g(v,X)(w)-g(w,X)(v)\,,
\end{equation}
while for \textbf{E} one computes
\begin{eqnarray}\label{eq:condition6}
 \beta(X,Y)(v)=-\bar\omega(f(v,X),Y)-\bar\omega(X,f(v,Y))\,.
\end{eqnarray}
Equation \eqref{eq:condition5} is equivalent to $\nu\hook\bar\omega=2\Alt(g)$ in $\Lambda^2 V^*\otimes \bar\mfg^*$, which we rewrite, according to our convention, as 
\begin{equation*}
\nu\hook\bar\omega=2A\,;
\end{equation*}
this is the first condition in the lemma. Viewing $f$ as an element of $V^*\otimes\Der(\bar\mfg)$, \eqref{eq:condition6} is nothing but 
\begin{equation*}
 \beta=-f.\bar\omega
\end{equation*}
in $\Lambda^2\bar\mfg^*\otimes V^*$, which is precisely the second condition in the statement.
\end{proof}


We introduce some conventions and notation before stating the last lemma.
\begin{itemize}
 \item Given $T\in V^*\otimes V^*\otimes W$, we define the \emph{trace} of $T$, $\tr(T)\in\Lambda^2 V^*\otimes W$, by using the canonical isomorphisms 
\[
V^*\otimes V^*\cong V^*\otimes V\otimes \Lambda^2 V^*\cong \End(V)\otimes\Lambda^2 V^*\,.
\]
\item Given $T\in V^*\otimes V^*\otimes \bar\mfg ^*$, define $T_I\in V^*\otimes V^*\otimes\bar\mfg ^*$ by $T_I(v,w)=T(Iv,w)$.
\item The symplectic form $\bar\omega$ gives an isomorphism $^\sharp\colon\bar\mfg^*\to\bar\mfg$ which can be extended to a map $^\sharp\colon V^*\otimes V^*\otimes\bar\mfg^*\to V^*\otimes V^*\otimes \bar\mfg$ 
\end{itemize}

\begin{lemma}\label{lemma-J}
Assume that $[\cdot,\cdot]$ defined by \eqref{eq:Liebracketong1}, \eqref{eq:Liebracketong2} and \eqref{eq:Liebracketong3} is a Lie bracket on $\mfg$
and that the $2$-form $\omega$ defined by \eqref{eq:omegaong} is closed. Then the almost complex structure $J$ defined by \eqref{eq:Jong} is integrable if and only if
\[
I^*f-\bar J\circ f\in V^*\otimes\mathfrak{sp}(\bar\mfg,\bar J,\bar\omega) \quad \text{and} \quad \nu=\bar J(\tr(S_I))^\sharp\,. 
\]
\end{lemma}

\begin{proof}
We only have to check the integrability for elements of $V\oplus\bar\mfg$ since $V^*$ is central and $J$-invariant.
Moreover, we always have $N_J(\tilde{X},J\tilde{X})=0$ for any $\tilde{X}\in \mfg$. Hence, since $V$
is $J$-invariant and two-dimensional, we are left with the following two cases:
\begin{description}
 \item[F] $(v,0,0)$, $(0,X,0)$ and
 \item[G] $(0,X,0)$, $(0,Y,0)$.
\end{description}

Applying formul\ae \ \eqref{eq:Jong} and \eqref{eq:Liebracketong2} to \textbf{F}, we get
\begin{eqnarray}
 f(v,X)+\bar J f(Iv,X)+\bar J f(v,\bar J X)-f(Iv,\bar J X) &=& 0\label{eq:condition7}\, ,\\
 g(v,X)+I^*\big(g(Iv,X)+g(v,\bar J X)\big)-g(Iv,\bar J X) &=& 0\label{eq:condition8}\,;
\end{eqnarray}
Using Lemma~\ref{lemma-jacobi}, we consider $f$ as an element of $V^*\otimes\Der(\bar\mfg)$ and define:
\begin{itemize}
 \item $I^*f\in V^*\otimes\Der(\bar\mfg)$ by $(I^*f)(v,X)=f(Iv,X)$;
 \item $\bar J\circ f\in V^*\otimes\End(\bar\mfg)$ by $(\bar J\circ f)(v,X)=\bar Jf(v,X)$.
\end{itemize}
One sees immediately that \eqref{eq:condition7} is equivalent to 
\begin{equation}\label{eq:condition7a}
 \left\{I^*f-\bar J\circ f,\bar J\right\}=0\in V^*\otimes\End(\bar\mfg),
\end{equation}
i.e. to $I^*f-\bar J\circ f\in V^*\otimes\End(\bar\mfg,\bar J)$. As for the second equation, given $v\in V$, note that $\{ v,Iv\}$ is a basis of $V$.
Evaluating \eqref{eq:condition8} on $v$ and $Iv$ we get
\begin{eqnarray}
 g(v,X,v)+g(Iv,X,Iv)+g(v,\bar J X,Iv)-g(Iv,\bar J X,v) &=& 0\label{eq:condition8a}\\
 g(v,X,Iv)-g(Iv,X,v)-g(v,\bar J X,v)-g(Iv,\bar J X,Iv) &=& 0\label{eq:condition8b}\,;
\end{eqnarray}
by plugging $\bar J X$ instead of $X$ in \eqref{eq:condition8b}, we recover \eqref{eq:condition8a}, thus both conditions are equivalent. Writing $g=S+A$ as before, \eqref{eq:condition8a} amounts to
\begin{align}
 0&=\big(S(v,v)+S(Iv,Iv)\big)X+\big(g(v,Iv)-g(Iv,v)\big)\bar J X\nonumber\\
 &=\big(S(v,v)+S(Iv,Iv)\big)X+2A(v,Iv)\bar J X\nonumber\\
 &=\big(S(v,v)+S(Iv,Iv)\big)X+\bar\omega\big(\nu(v,Iv),\bar J X\big)\nonumber\\
 &=\big(S(v,v)+S(Iv,Iv)\big)X+\bar\omega\big(\bar J\nu(v,Iv),X\big)\label{eq:condition8c}\,.
\end{align}
According to our conventions and using~\eqref{eq:condition8c}, one checks that
\[
\tr(S_I)(v,Iv)=S(v,v)+S(Iv,Iv)=-\bar\omega(\bar J\nu(v,Iv),\cdot)\,.
\]
Using the isomorphism given by the symplectic form, we can rewrite the last equation as
\begin{equation*}
\nu=\bar J(\tr(S_I))^\sharp\,.
\end{equation*}
This ensures \eqref{eq:condition8c} and gives the second condition in our statement.

For \textbf{G}, \eqref{eq:Jong} gives two conditions; the first one is $N_{\bar J}(X,Y)=0$, which is satisfied, since $\bar J$ is a complex structure on $\bar\mfg$. The second one reads
\begin{equation*}
 \beta(X,Y)+I^*\beta(\bar J X,Y)+I^*\beta(X,\bar J Y)-\beta(\bar J X,\bar J Y)=0\,;
\end{equation*}
in view of the second expression in Lemma~\ref{lemma-omega}, this gives
\begin{align*}
 0&=(f.\bar\omega)(X,Y)+(f.\bar\omega)(\bar J X,Y)\circ I+(f.\bar\omega)(X,\bar J Y)\circ I-(f.\bar\omega)(\bar J X,\bar J Y)\\
 &=(f.\bar\omega)(X,Y)+\big((I^*f).\bar\omega\big)(\bar J X,Y)+\big((I^*f).\bar\omega\big)(X,\bar J Y)-(f.\bar\omega)(\bar J X,\bar J Y)\\
 &=\big((I^*f-\bar J\circ f).\bar\omega\big)(\bar J X,Y)+\big((I^*f-\bar J\circ f).\bar\omega\big)(X,\bar J Y)\,,
\end{align*}
where the symmetry of $\bar J$ with respect to $\bar\omega$ was used in the last equality. Moreover, using again the symmetry of $\bar J$ and \eqref{eq:condition7a} one sees that the two terms in the last equation above are equal. Since this holds for arbitrary vectors, we get $(I^*f-\bar J\circ f).\bar\omega=0$ which, together with \eqref{eq:condition7a}, gives
the first condition in the lemma:
\begin{equation*}
I^*f-\bar J\circ f\in V^*\otimes\mathfrak{sp}(\bar\mfg,\bar J,\bar\omega)\,.
\end{equation*}
\end{proof}

With these three lemmas we prove the following result:

\begin{proposition}\label{pro:complexsymplecticoxidation}
Let $(\bar{\mfg},\bar J,\bar{\omega})$ be a complex symplectic Lie algebra, let $V$ be a two-dimensional real vector space with a complex structure
 $I\colon V\to V$ and set $\mfg\coloneqq V\oplus\bar{\mfg}\oplus V^*$. Let $\left[\cdot,\cdot\right]\colon\mfg\times\mfg\to\mfg$ be the skew-symmetric map defined by equations \eqref{eq:Liebracketong1} -- \eqref{eq:Liebracketong3}, let $J\colon\mfg\to\mfg$ be the almost complex structure defined by \eqref{eq:Jong} and let $\omega\in\Lambda^2\mfg^*$ be the 2-form defined by \eqref{eq:omegaong}. Also, let $g=S+A$ be the decomposition of $g$ in its symmetric and skew-symmetric part. Then $\left[\cdot,\cdot\right]$ is a Lie bracket on $\mfg$ and $(\mfg,J,\omega)$ a complex symplectic Lie algebra if and only if all of the following seven conditions hold:
\begin{itemize}
\item[1.] $f\in V^*\otimes \mathrm{Der}(\bar{\mfg})$ and $I^*f-\bar{J}\circ f\in V^*\otimes\mathfrak{sp}(\bar{\mfg},\bar J,\bar{\omega})$;
\item[2.] $\beta=-f.\bar\omega$;
\item[3.] $\nu=\bar J(\tr(S_I))^{\sharp}$;
\item[4.] $A=\frac{1}{2}\nu\hook\bar\omega=\frac{1}{2} \bar J^*(\tr(S_I))$;
\item[5.] $\Alt(S\circ f)=-\frac{1}{2}\bar\omega(f(\nu),\cdot)-\frac{3}{4}\bar\omega(\nu,f)$;
\item[6.] $\{f,f\}=\ad_\nu$;
\item[7.] $d_{\bar\mfg} S=f.f.\bar{\omega}-\tfrac{1}{2}d_{\bar\mfg}(\bar J^*(\tr_V(S_I)))$.           
\end{itemize}
\end{proposition}

\begin{proof}
Observe that Lemmas~\ref{lemma-jacobi}, \ref{lemma-omega} and \ref{lemma-J} must hold. These directly give us conditions 1 -- 4 and 6 in the statement above. We next show that 5. and 7. are equivalent to conditions \textrm{(iii)} and \textrm{(iv)} in Lemma~\ref{lemma-jacobi}. This will conclude the proof, as part \textrm{(v)} in Lemma~\ref{lemma-jacobi} is trivially fulfilled. The reason is that
conditions 1. and 2. in the statement
(namely, $f\in V^*\otimes \mathrm{Der}(\bar{\mfg})$ and
$\beta=-f.\bar\omega$) directly give
\[
 d_{\bar\mfg}\beta=-d_{\bar\mfg}(f.\bar\omega)=-f.(d_{\bar\mfg}\bar\omega)=0,
\]
since $\bar\omega$ is $d_{\bar\mfg}$-closed.

Let us then rewrite part \textrm{(iv)} of Lemma~\ref{lemma-jacobi}. On the one hand, using part 2. of the proposition one gets
\begin{align*}
\Alt(g\circ f)(v,w,X,u)&=-\frac{1}{2}(\nu\hook f.\bar\omega)(v,w,X,u)\\
&=-\frac{1}{2}\bar\omega\big(f(u,\nu(v,w)),X\big)-\frac{1}{2}\bar\omega\big(\nu(v,w),f(u,X)\big).
\end{align*}
On the other hand, by 4. in the proposition, we get
\begin{align*}
\Alt(A\circ f)(v,w,X,u)&=\frac{1}{2}\,A\big(v,f(w,X),u\big)-\frac{1}{2}\,A\big(w,f(v,X),u\big)\\
&=\frac{1}{4}\,(\nu\hook\bar\omega)\big(v,f(w,X),u\big)-\frac{1}{4}\,(\nu\hook\bar\omega)\big(w,f(v,X),u\big)\\
&=\frac{1}{4}\,\bar\omega\big(\nu(v,u),f(w,X)\big)-\frac{1}{4}\,\bar\omega\big(\nu(w,u),f(v,X)\big)\,.
\end{align*}
Hence, writing $g=S+A$, we have
\begin{align*}
\Alt(S\circ f)(v,w,X,u)&=-\frac{1}{2}\,\bar\omega\big(f(u,\nu(v,w)),X\big)-\frac{3}{4}\,\bar\omega\big(\nu(v,w),f(u,X)\big)\\
&+\frac{1}{4}\,\Big(\bar\omega\big(\nu(v,w),f(u,X)\big)-\bar\omega\big(\nu(v,u),f(w,X)\big)\\
&+\bar\omega\big(\nu(w,u),f(v,X)\big)\Big)\\
&=-\frac{1}{2}\,\bar\omega\big(f(u,\nu(v,w)),X\big)-\frac{3}{4}\,\bar\omega\big(\nu(v,w),f(u,X)\big)\,,
\end{align*}
since the last three terms form an element of $\Lambda^3V^*\otimes\bar\mfg^*$, which is zero since $\dim V~=~2$. We rewrite the last equation as
\begin{equation*}
 \Alt(S\circ f)=-\frac{1}{2}\,\bar\omega\big(f(\nu),\cdot\big)-\frac{3}{4}\,\bar\omega\big(\nu,f\big)
\end{equation*}
in $\Lambda^2V^*\otimes\bar\mfg^*\otimes V^*$. This is precisely part  5. in the statement above.

We now focus on part \textrm{(iii)} of Lemma~\ref{lemma-jacobi}. Using conditions 3. and 4. we compute
\begin{align*}
(d_{\bar\mfg}A)(v,X,Y,w)&=-A(v,[X,Y]_{\bar\mfg},w)=-\frac{1}{2}\bar\omega(\nu(v,w),[X,Y]_{\bar\mfg})\\
&=-\frac{1}{2}\bar J^*\tr(S_I)(v,w,[X,Y]_{\bar\mfg})=\frac{1}{2}d_{\bar\mfg}(\bar J^*\tr(S_I)(v,w))(X,Y)
\end{align*}
and we arrive at
\begin{equation}\label{eq:condition100b}
d_{\bar\mfg}A=\frac{1}{2}d_{\bar\mfg}(\bar J^*\tr(S_I))\,.
\end{equation}
Moreover, from part 2. we have
\begin{equation}\label{eq:condition100a}
d_{\bar\mfg}g=-f.\beta=f.f.\bar\omega\,.
\end{equation}
Putting \eqref{eq:condition100a} and \eqref{eq:condition100b} together, we obtain
\begin{equation*}
d_{\bar\mfg}S=d_{\bar\mfg}g-d_{\bar\mfg}A=f.f.\bar\omega-\frac{1}{2}d_{\bar\mfg}(\bar J^*\tr(S_I))\,.
\end{equation*}
which is precisely condition 7. in our statement.
\end{proof}

Let $(\bar{\mfg},\bar J,\bar{\omega})$ be a complex symplectic Lie algebra and let $V$ be 2-dimensional real vector space endowed with a complex structure $I$. Given a triple $(f,S,\tau)\in V^*\otimes \mathrm{Der}(\bar{\mfg})\times S^2 V^*\otimes \bar{\mfg}^*\times \Lambda^2 V^*\otimes V^*$, define
\begin{equation}\label{eq:inducedtensorsoxdata}
\beta\coloneqq -f.\bar{\omega},\qquad \nu\coloneqq \bar J(\tr(S_I))^{\sharp}, \qquad 
A\coloneqq\frac{1}{2} \bar J^*(\tr(S_I))\,.
\end{equation}
Then Proposition \ref{pro:complexsymplecticoxidation} motivates the following definition:

\begin{definition}\label{def:oxidationdata}
A triple $(f,S,\tau)$ is called \textit{complex symplectic oxidation data} for $(\bar{\mfg},\bar J,\bar{\omega})$ if
\begin{itemize}
\item[(i)] $I^*f-\bar J\circ f\in V^*\otimes\mathfrak{sp}(\bar{\mfg},\bar J,\bar{\omega})$
\item[(ii)] $\{f,f\}=\ad_\nu$,
\item[(iii)] $d_{\bar\mfg} S=f.f.\bar{\omega}-\tfrac{1}{2}d_{\bar\mfg}(\bar J^*(\tr_V(S_I)))$ and
\item[(iv)] $\Alt(S\circ f)=-\frac{1}{2}\bar\omega(f(\nu),\cdot)-\frac{3}{4}\bar\omega(\nu,f)$.
\end{itemize}
The complex symplectic Lie algebra $(\mfg,\omega,J)$ defined by \eqref{eq:omegaong} -- \eqref{eq:Liebracketong3}
using the tensors defined in \eqref{eq:inducedtensorsoxdata} is the \emph{complex symplectic oxidation} of $(\bar{\mfg},\bar J,\bar{\omega})$ by the complex symplectic oxidation data $(f,S,\tau)$. 
Two complex symplectic oxidation data $(f,S,\tau)$ and $(\tilde{f},\tilde{S},\tilde{\tau})$ are \emph{isomorphic} if the corresponding complex symplectic oxidations are isomorphic as complex symplectic Lie algebras.
\end{definition}
\begin{remark}
Note that the natural action of $\GL(V,I)\times \mathrm{Aut}(\bar{\mfg},\bar J,\bar{\omega})$ on
$V^*\otimes \mathrm{Der}(\bar{\mfg})\times S^2 V^*\otimes \bar{\mfg}^*\times \Lambda^2 V^*\otimes V^*$ preserves the subset of 
complex symplectic oxidation data and all elements in one orbit are isomorphic to each other.
\end{remark}

We extract the following two corollaries from Proposition \ref{pro:complexsymplecticoxidation}:
\begin{corollary}\label{co:redox}
The complex symplectic oxidation $(\mfg,J,\omega)$ of a $4n$-dimensional complex symplectic Lie algebra $(\bar{\mfg},\bar J,\bar{\omega})$ by complex symplectic oxidation data $(f,S,\tau)$ is a complex symplectic Lie algebra of dimension $4n+4$ with a central, $J$-invariant ideal $V^*$. The complex symplectic reduction of $(\mfg,J,\omega)$ with respect to $V^*$ is isomorphic to $(\bar{\mfg},\bar J,\bar{\omega})$. Moreover, $\mfg$ is nilpotent if $\bar{\mfg}$ and $f$ are nilpotent.
\end{corollary}

\begin{corollary}\label{co:redox-2}
A complex symplectic Lie algebra $(\mfg,J,\omega)$ of dimension $4n+4$ is the oxidation of a complex symplectic Lie algebra $(\bar{\mfg},\bar J,\bar{\omega})$ of dimension $4n$ if and only if $\mfz(\mfg)$ satisfies
\[
 \mathfrak{z}(\mfg)\cap J\mathfrak{z}(\mfg)\neq \{0\}\,.
\]

\end{corollary}

For later purposes, it is useful to reformulate the conditions on the complex symplectic oxidation data $(f,S,\tau)$ 
by fixing a basis $\{v,Iv\}$ of $(V,I)$:

\begin{lemma}\label{le:oxidationdata}
Let $(\bar{\mfg},\bar J,\bar{\omega})$ be a complex symplectic Lie algebra, let $V$ be a two-dimensional real vector
space endowed with a complex structure $I$ and $v\in V\setminus \{0\}$ be given. Moreover, let $(f,S,\tau)\in V^*\otimes \mathrm{Der}(\bar{\mfg})\times S^2 V^*\otimes \bar{\mfg}^*\times \Lambda^2 V^*\otimes V^*$ be given and set $v_1\coloneqq v$, $v_2\coloneqq Iv$, $f_i\coloneqq f(v_i)\in \mathrm{Der}(\bar{\mfg})$
and $S_{ij}\coloneqq S(v_i,v_j)\in \bar{\mfg}^*$ for all $i,j=1,2$. Then $(f,S,\tau)$ is a complex symplectic oxidation datum on $(\bar{\mfg},\bar J,\bar{\omega})$ if and only if 
\begin{itemize}
\item[(i)] $f_2-\bar J f_1\in \mathfrak{sp}(\bar{\mfg},\bar J,\bar{\omega})$,
\item[(ii)]
$\{f_1,f_2\}=[\bar J(S_{11}^{\sharp}+S_{22}^{\sharp}),\cdot]_{\bar{\mfg}}$,
\item[(iii)]
$d_{\bar{\mfg}} S_{ii}=f_i.(f_i.\bar{\omega})$ for $i=1,2$, 
$d_{\bar{\mfg}} S_{12}=f_1.(f_2.\bar{\omega})-\tfrac{1}{2}d_{\bar{\mfg}} (\bar J^*(S_{11}+S_{22}))$ and
\item[(iv)]
$S_{i2}\circ f_1-S_{i1}\circ f_2=\,\bar{\omega}\left(f_i\left(\bar J(S_{11}^{\sharp}+S_{22}^{\sharp})\right),\cdot\right)
                                 +\frac{3}{2}\,\left(S_{11}+S_{22}\right)\circ \bar J\circ f_i$ for $i=1,2$.
\end{itemize}
\end{lemma}
\begin{proof}
We show the equivalence of any of the conditions (i)-- (iv) in Lemma \ref{le:oxidationdata} to the condition with the same number in Definition \ref{def:oxidationdata}.

The equivalence of the conditions (i) is clear once one notices that
$(I^*f-\bar J\circ f)(v_1)=f_2-\bar J f_1\in \mathfrak{sp}(\bar{\mfg},\bar J,\bar{\omega})$ implies
$(I^*f-\bar J\circ f)(v_2)=-\bar J\circ (f_2-\bar J\circ f_1)\in \mathfrak{sp}(\bar{\mfg},\bar J,\bar{\omega})$.

The conditions in (ii) are equivalent, since 
\[
\nu(v_1,v_2)=\bar J(\tr(S_I))^{\sharp} (v_1,v_2)=\bar J(\tr(S_I)(v,Iv))^{\sharp}=\bar J(S_{11}^{\sharp}+S_{22}^{\sharp})\,. 
\]

As the conditions (iii) are obviously equivalent, we are left with showing the equivalence of conditions (iv). But that equivalence follows from
\begin{equation*}
\begin{split}
\Alt(S\circ f)(v_1,v_2,X,v_i)&=\tfrac{1}{2}\left(S(v_1,f(v_2,X),v_i)-S(v_2,f(v_1,X),v_i)\right)\\
&=\tfrac{1}{2}(S_{1i}\circ f_2-S_{2i}\circ f_1)(X)\\
&=-\tfrac{1}{2}(S_{i2}\circ f_1-S_{i1}\circ f_2)(X),
\end{split}
\end{equation*}

\begin{equation*}
\bar\omega(f(\nu),\cdot)(v_1,v_2,X,v_i)=\bar\omega(f(v_i,\nu(v_1,v_2)),X)=\bar{\omega}\left(f_i\left(\bar J(S_{11}^{\sharp}+S_{22}^{\sharp})\right),X\right)
\end{equation*}
and
\begin{equation*}
\begin{split}
\bar\omega(\nu,f)(v_1,v_2,X,v_i)&=\bar\omega(\nu(v_1,v_2),f(v_i,X))=\bar\omega(\bar J(S_{11}^{\sharp}+S_{22}^{\sharp}),f_i(X))\\
&=\bar\omega(S_{11}^{\sharp}+S_{22}^{\sharp},\bar J f_i(X))=\left(\left(S_{11}+S_{22}\right)\circ \bar J\circ f_i\right)(X)
\end{split}
\end{equation*}
for any $X\in\overline{\mfg}$ and any $i\in \{1,2\}$.
\end{proof}

\section{Nilpotent complex symplectic Lie algebras}\label{explicit-8-d}

In this section we apply our construction to nilpotent Lie algebras and show that every nilpotent complex symplectic Lie algebra of dimension $4$ and $8$ can be obtained by oxidating a nilpotent complex symplectic Lie algebras of dimension $0$ and $4$ respectively. We start with some preliminary facts concerning complex structures on nilpotent Lie algebras.


\begin{definition}
Let $\mfg$ be a Lie algebra. The \textit{ascending central series} of $\mfg$ is the sequence of ideals of $\mfg$ recursively defined by 
\[
\mfg_0=\left\lbrace 0\right\rbrace, \quad \mfg_k=\left\lbrace X\in\mfg\ |\ \left[X,\mfg\right]\subseteq\mfg_{k-1} \right\rbrace \textrm{ for } k\geq 1\,.
\]
\end{definition}
Notice that $\mfg_1=\mfz(\mfg)$ and that $\mfg$ is nilpotent if there is some $s\geq 1$ such that $\mfg_k=\mfg$ for all $k\geq s$. The smallest integer $s$ satisfying the latter condition is called the
\textit{nilpotency step} of $\mfg$ and then the Lie algebra $\mfg$ is said to be $s$-\textit{step nilpotent}. 
\begin{definition}
Let $\mfg$ be an $s$-step nilpotent Lie algebra $\mfg$ and set $m_k=\dim\mfg_k$ for $1\leq k\leq s$. The \emph{ascending type} of $\mfg$ is $(m_1,\ldots,m_s)$.
\end{definition}
Let $J\colon\mfg\to\mfg$ be a complex structure on the nilpotent Lie algebra $\mfg$. In general, the ideals $\mfg_k$ in the ascending central series are not $J$-invariant, see \cite{Cordero-Fernandez-Gray-Ugarte}. In order to adapt the ascending central series to the existence of a complex structure $J$, the following definitions are introduced in \cite{Cordero-Fernandez-Gray-Ugarte, Latorre-Ugarte-Villacampa}.
\begin{definition}
Let $(\mfg,J)$ be a nilpotent Lie algebra endowed with a complex structure. The \emph{ascending $J$-compatible series} $\left\lbrace \mfa_k(J)\right\rbrace_k$ of $\mfg$ is recursively defined by $\mfa_0(J)=\left\lbrace 0\right\rbrace$ and
\[
\mfa_k(J)=\left\lbrace X\in\mfg\ |\ \left[X,\mfg\right]\subseteq\mfa_{k-1}(J), \left[JX,\mfg \right]\subseteq \mfa_{k-1}(J) \right\rbrace \textrm{ for }k\geq 1\,.
\]
\end{definition}
Now we are ready to define different natural classes of complex structures on nilpotent Lie algebras according to how the ascending $J$-compatible series
looks like:
\begin{definition}
A complex structure $J$ on a nilpotent Lie algebra $\mfg$ is said to be
\begin{itemize}
  \item[(i)] \textit{strongly non-nilpotent} (SnN) if $\mfa_1(J)=\left\lbrace 0\right\rbrace$;
  \item[(ii)] \textit{quasi-nilpotent} if it satisfies $\mfa_1(J)\neq\left\lbrace 0\right\rbrace$; moreover $J$ will be called
  \begin{itemize}
    \item[(ii.1)] \textit{nilpotent}, if there exists some integer $t>0$ such that $\mfa_t(J)=\mfg$,
    \item[(ii.2)] \textit{weakly non-nilpotent}, if there is an integer $t>0$ such that $\mfa_\ell(J)=\mfa_t(J)$ for all $\ell\geq t$, and $\mfa_t(J)\neq\mfg$.
  \end{itemize}
\end{itemize}
\end{definition}

As explained in \cite{Latorre-Ugarte-Villacampa}, quasi-nilpotent complex structures on a nilpotent Lie algebra of a given dimension can be reduced to complex structures on a nilpotent Lie algebra of lower dimension; in particular, notice that $\mfa_1(J)=\mfz(\mfg)\cap J\mfz(\mfg)$ is central and $J$-invariant. On the contrary, strongly non-nilpotent complex structures are ``genuine'' in each dimension, and constitute the building blocks of complex structures on nilpotent Lie algebras of higher dimension.

The following result contained in \cite{Cordero-Fernandez-Gray-Ugarte}
gives a characterization of complex structures of nilpotent type:

\begin{thm}\label{nilpotent-J}
Let $(\mfg,J)$ be a $2n$-dimensional nilpotent Lie algebra endowed with a complex structure. $J$ is nilpotent if and only if there is a basis $\{\varphi^i\}_{i=1}^n$ of $\mfg^{*(1,0)}$ such that the complex structure equations of $(\mfg,J)$ satisfy
\[
d\varphi^i=\sum_{j<k<i} A^i_{jk}\,\varphi^j\wedge\varphi^k
+ \sum_{j,k<i} B^i_{jk}\,\varphi^j\wedge\overline{\varphi^k}, 
\quad 1\leq i \leq n\,, 
\]
where $A^i_{jk}$, $B^i_{jk}$ are complex numbers.
\end{thm}

For the sake of simplicity, from now on we will denote by $\varphi^{jk}$ the wedge product $\varphi^j\wedge\varphi^k$
and by~$\varphi^{\overline{k}}$ the complex conjugate $\overline{\varphi^k}$.

The next corollary summarizes the extent to which reduction and oxidation allow to describe complex symplectic structures on nilpotent Lie algebras.
\begin{corollary}\label{co:redox_nilpotent}
Every complex symplectic nilpotent Lie algebra $(\mfg,J,\omega)$ of dimension $4n+4$ with quasi-nilpotent complex structure $J$ is the oxidation of a complex symplectic Lie algebra $(\bar\mfg,\bar J,\bar\omega)$ of dimension $4n$.
\end{corollary}
\begin{proof}
 This follows immediately from Corollary \ref{co:redox-2} and the fact that if $J$ is quasi-nilpotent, then $\mfz(\mfg)\cap J\mfz(\mfg)=\mfa_1(J)\neq 0$.
\end{proof}


Our next goal is to show that if $(\mfg,J,\omega)$ is a complex symplectic nilpotent Lie algebra of dimension $4$ and $8$, then $J$ is necessarily quasi-nilpotent.

\subsection{Nilpotent complex symplectic Lie algebras in dimension four}
First of all, note that there are only three four-dimensional nilpotent Lie algebras,
namely, $\bR^4$, $\mfh_3\oplus \bR$ and $\mfn_4$. The non-zero Lie brackets of $\mfh_3\oplus \bR$ and $\mfn_4$ are, in terms of a basis $\{e_1,\ldots,e_4\}$, $[e_1,e_2]=e_3$ and $[e_1,e_4]=e_2,\, [e_2,e_4]=e_3$, respectively. Now \cite{Snow} shows that of these only $\bR^4$ and $\mfh_3\oplus \bR$ have an integrable almost complex structure, which is given, up to equivalence, in both cases by $J_0=e^2\otimes e_1-e^1\otimes e_2+ e^4\otimes e_3-e^3\otimes e_4$, where
$\{e^1,\ldots,e^4\}$ denotes the dual basis of $\{e_1,\ldots,e_4\}$. Next, observe that any
$\omega\in \Lambda^2\mfg^*$, $\mfg\in\{\bR^4,\mfh_3\oplus \bR\}$, for which $(J_0,\omega)$ is a complex
symplectic Lie algebra, is a non-degenerate real form of type $(2,0)$ and $(0,2)$ and so given by
\begin{equation*}
\omega:=\omega(z):=\Re(z) (e^{14}+e^{23})+\Im(z)(e^{13}-e^{24})
\end{equation*}
for some $z\in\bC\setminus \{0\}$. Note that $\omega(z)$ is closed for all values of $z$ in both cases.
For $w\in \bC\setminus \{0\}$, 
consider the vector space automorphism $F(w)$ of $\mfg$ given with respect to the above basis by
\begin{equation*}
F(w):=\left(\begin{smallmatrix} \Re(w) &  - \Im(w) & 0 & 0\\ \Im(w) &  \Re(w) & 0 & 0\\  0 & 0 & \left|w\right|^2 & 0 \\ 0 & 0 &  0 &\left|w\right|^2 \\\end{smallmatrix}\right).
\end{equation*}
One easily checks that $F(w)$ is in both cases a Lie algebra automorphism, preserves $J_0$ and that
$F(w)^* \omega(z)=\omega(wz \left|w\right|^2)$. Hence, one may find $w\in \bC\setminus \{0\}$ such that 
\begin{equation*}
F(w)^*\omega(z)=\omega(1)=e^{14}+e^{23}=:\omega_0
\end{equation*}
i.e. $(\mfg,J_0,\omega(z))$ is isomorphic to $(\mfg,J_0,\omega_0)$. Thus, with respect to the above basis $\{e_1,\ldots,e_4\}$ of $\mfg$, 
we have shown:
\begin{proposition}\label{pro:4dcomplexsymplectic}
Let $(\mfg,J,\omega)$ be a four-dimensional nilpotent complex symplectic Lie algebra. Then $(\mfg,J,\omega)$ is isomorphic to $(\mfh_3\oplus \bR,J_0,\omega_0)$
or to $(\bR^4,J_0,\omega_0)$ with
\begin{equation*}
 J_0=e^2\otimes e_1-e^1\otimes e_2+ e^4\otimes e_3-e^3\otimes e_4 \quad \textrm{and} \quad \omega_0=e^{14}+e^{23}.
\end{equation*}
\end{proposition}

\begin{remark}
The subspace $\la e_3,e_4\ra$ is $J_0$-invariant and central both in $\bR^4$ and $\mfh_3\oplus\bR$. By Corollary \ref{co:redox_nilpotent}, one can obtain $(\bR^4,J_0,\omega_0)$ and $(\mfh_3\oplus \bR,J_0,\omega_0)$ from the trivial complex symplectic Lie algebra $\{0\}$. More precisely, the complex symplectic oxidation data $(f,S,\tau)$ on $\{0\}$ always satisfies $f=0$ and $S=0$ and one gets, up to isomorphism, $(\bR^4,J_0,\omega_0)$ for $\tau=0$ and $(\mfh_3\oplus \bR,J_0,\omega_0)$ for $\tau\neq 0$.
\end{remark}

\subsection{Nilpotent complex symplectic Lie algebras in dimension eight}
In \cite{Latorre-Ugarte-Villacampa}, the authors obtained the following result on 8-dimensional nilpotent Lie algebras endowed with a strongly non-nilpotent complex structure:

\begin{thm}\label{SnN-8-d}
Let $J$ be a strongly non-nilpotent complex structure on an 8-dimen\-sional nilpotent Lie algebra $\mfg$. Then
\begin{itemize}
\item $\dim\mfz(\mfg)=1$;
\item the ascending type of $\mfg$ belongs to the following list:
\begin{gather*}
(1,3,8), \quad (1,3,5,8), \quad (1,3,5,6,8), \quad (1,3,6,8), \quad (1,4,8),\\ 
(1,4,6,8), \quad (1,5,8), \quad \textrm{or} \quad (1,5,6,8).
\end{gather*}
\end{itemize}

Moreover, there exists a basis $\{\varphi^1, \varphi^2, \varphi^3, \varphi^4\}$ of $\mfg^{*(1,0)}$ in terms of which the complex structure equations of $(\mfg,J)$ are one of the following:
\begin{itemize}
\item[(i)] if the ascending type of $\mfg$ is $(1,3,8)$, $(1,3,6,8)$, $(1,3,5,8)$ or $(1,3,5,6,8)$, then
\[\arraycolsep=1.4pt\def\arraystretch{1.2}
\left\{\begin{array}{rcl}
d\varphi^1 &=& 0,\\
d\varphi^2 &=& A\,\varphi^{1\bar 1} -B(\varphi^{14}- \varphi^{1\bar 4}),\\
d\varphi^3 &=& C\,\varphi^{12}-E\,(\varphi^{14}- \varphi^{1\bar 4})\\
  	&&
	+F\,\varphi^{1\bar 1}+D\,\varphi^{1\bar 2} - H\,(\varphi^{24}-\varphi^{2\bar 4}) +G\,\varphi^{2\bar1}+K\,\varphi^{2\bar 2},\\
d\varphi^4 &=& L\,\varphi^{1\bar 1}+M\,\varphi^{1\bar 2}+N\,\varphi^{1\bar 3}\\
  && -\bar{M}\,\varphi^{2\bar 1} +i\,s\,\varphi^{2\bar 2} 
  + P\,\varphi^{2\bar 3}-\bar{N}\,\varphi^{3\bar 1}-\bar{P}\,\varphi^{3\bar2};
 \end{array}\right.
\]


\medskip

\item[(ii)]  if the ascending type is $(1,4,8)$ or $(1,4,6,8)$, then
\[\arraycolsep=1.4pt\def\arraystretch{1.2}
\left\{\begin{array}{rcl}
d\varphi^1 &=& 0,\\
d\varphi^2 &=& A\,\varphi^{1\bar 1},\\
d\varphi^3 &=& -D (\varphi^{12}-\varphi^{1\bar 2})-E\,(\varphi^{14}-\varphi^{1\bar 4})+F\,\varphi^{1\bar 1},\\
d\varphi^4 &=& L\,\varphi^{1\bar 1}+M\,\varphi^{1\bar 2}+N\,\varphi^{1\bar 3}-\bar{M}\,\varphi^{2\bar 1}
  +i\,s\,\varphi^{2\bar 2}-\bar{N}\,\varphi^{3\bar 1};
\end{array}\right.
\]


\item[(iii)]  if the ascending type is $(1,5,8)$ or $(1,5,6,8)$, then
\[\arraycolsep=1.4pt\def\arraystretch{1.2}
\left\{\begin{array}{rcl}
d\varphi^1  &=& 0,\\
d\varphi^2  &=& A\,\varphi^{1\bar 1} -B(\varphi^{14}- \varphi^{1\bar 4}),\\
d\varphi^3  &=& F\,\varphi^{1\bar 1} -E(\varphi^{14}- \varphi^{1\bar 4}),\\
d\varphi^4  &=& L\,\varphi^{1\bar 1}+M\,\varphi^{1\bar 2}+N\,\varphi^{1\bar 3}\\
	&& -\bar{M}\,\varphi^{2\bar 1}+i\,s\,\varphi^{2\bar 2}+P\,\varphi^{2\bar 3} 
  	-\bar{N}\,\varphi^{3\bar 1} -\bar P\,\varphi^{3\bar 2}+i\,t\,\varphi^{3\bar 3}.
\end{array}\right.
\]

\end{itemize}
All coefficients are complex numbers except for $s, t$, which are real. Additionally the coefficients must be such that $d^2=0$, which is equivalent to the Jacobi identity.
\end{thm}

\begin{remark}\label{nec-but-not-suf}
Mind that the presence of an SnN complex structure $J$ on $\mfg$ implies the existence of a basis for $\mfg^{*(1,0)}$ satisfying part \textrm{(i)}, \textrm{(ii)}, or \textrm{(iii)} of the previous theorem. However, the converse is not true: there are choices of the parameters that might give rise to nilpotent complex structures, according to Theorem~\ref{nilpotent-J}.
\end{remark}

From this parametrization of SnN complex structures on 8-dimensional nilpotent Lie algebras, one can deduce the following result:

\begin{proposition}\label{hs-implica-nilp}
An 8-dimensional nilpotent Lie algebra endowed with an SnN complex structure does not admit a complex symplectic structure.
\end{proposition}

\begin{proof}
A complex symplectic structure $\omega_\bC$ on an $8$-dimensional nilpotent Lie algebra $\mfg$ endowed
with a complex structure $J$ is given by a $(2,0)$-form, i.e.~an element of $\Lambda^2\mfg^{*(1,0)}$; hence we have
\begin{equation*}
\omega_\bC=\alpha\,\varphi^{12}+\beta\,\varphi^{13}+\gamma\,\varphi^{14}+\tau\,\varphi^{23}
    +\theta\,\varphi^{24}+\zeta\,\varphi^{34},
\end{equation*}
for $\alpha,\beta,\gamma,\tau,\theta,\zeta\in\mathbb C$, $d\omega_\bC=0$ and $\omega_\bC\wedge\omega_\bC\neq 0$. We simply observe that
\begin{equation}\label{no-deg}
\omega_\bC\wedge\omega_\bC\neq 0 \quad \Leftrightarrow \quad \alpha\,\zeta-\beta\,\theta+\tau\,\gamma\neq 0.
\end{equation}

Let us study the condition $d\omega_\bC=0$ when~$(\mfg,J)$ is parametrized
by the complex structure equations given in Theorem~\ref{SnN-8-d}. We will see that
the presence of~$\omega_\bC$ on~$(\mfg,J)$ implies the existence of a basis 
for~$\mfg^{*(1,0)}$ in terms of which the complex structure equations follow 
Theorem~\ref{nilpotent-J}, thus $J$ being nilpotent.

We first focus on those $(\mfg,J)$ parametrized by part \textrm{(i)} of the theorem. 
Since we are interested in SnN complex structures, we can assume
$$(B,E,H)\neq (0,0,0),$$
as otherwise Theorem~\ref{nilpotent-J} directly holds. 
Using the structure equations \textrm{(i)}, we compute
\begin{align*}
 d\omega_\bC &= (H\,\beta + C\,\zeta - E\,\tau)\,\varphi^{124}
	- (G\,\beta - L\,\theta - F\,\tau - \bar M\,\gamma)\,\varphi^{12\bar1} \\
& - (K\,\beta + i\,s\,\gamma - M\,\theta - D\,\tau)\,\varphi^{12\bar2} 
	- (P\,\gamma - N\,\theta)\,\varphi^{12\bar3}- (H\,\beta - E\,\tau)\,\varphi^{12\bar4} \\
&  + B\,\tau\,\varphi^{134} 
	+ (L\,\zeta - A\,\tau + \bar N\,\gamma)\,\varphi^{13\bar1}+ (M\,\zeta + \bar P\,\gamma)\,\varphi^{13\bar2} 
	+ N\,\zeta\,\varphi^{13\bar3}\\
& - B\,\tau\,\varphi^{13\bar 4} - (F\,\zeta + A\,\theta)\,\varphi^{14\bar1} - D\,\zeta\,\varphi^{14\bar2} - (E\,\zeta + B\,\theta)\,\varphi^{14\bar4}  
	\\
& - (\bar M\,\zeta - \bar N\,\theta)\,\varphi^{23\bar1}+ (i\,s\,\zeta + \bar P\,\theta)\,\varphi^{23\bar2} + P\,\zeta\,\varphi^{23\bar3} 
	- G\,\zeta\,\varphi^{24\bar 1}\\
&- K\,\zeta\,\varphi^{24\bar2} - H\,\zeta\,\omega^{24\bar4}\,.
\end{align*}
Imposing $d\omega_\bC=0$, the coefficient of each element $\varphi^{ijk}$ in the expression above, 
that we will denote $[d\omega_\bC]_{ijk}$, must be zero.
In particular, 
\[
[d\omega_\bC]_{13\bar3}=[d\omega_\bC]_{14\bar2}=
[d\omega_\bC]_{23\bar3}=[d\omega_\bC]_{24\bar1}=[d\omega_\bC]_{24\bar2}=
[d\omega_\bC]_{24\bar4}=0\,, 
\]
so we have 
\[
D\zeta=G\zeta=H\zeta=K\zeta=N\zeta=P\zeta=0\,,
\]
and two cases arise. If we assume $\zeta\neq 0$, then $D=G=H=K=N=P=0$ and from
$[d\omega_\bC]_{13\bar2}=[d\omega_\bC]_{23\bar2}=0$ we get $M=s=0$. Replacing
these values in Theorem~\ref{SnN-8-d}, part~\textrm{(i)}, it suffices to define
a new basis 
\[
\tau^1=\varphi^1,\quad \tau^2=\varphi^4,\quad \tau^3=\varphi^2,\quad \tau^4=\varphi^3\,, 
\]
to check that Theorem~\ref{nilpotent-J} is fulfilled. If $\zeta=0$, the non-degeneracy condition \eqref{no-deg} becomes
$\beta\theta-\gamma\tau\neq 0$. In particular, this implies $(\theta,\gamma)\neq (0,0)$
and $(\theta,\tau)\neq (0,0)$. Consequently, from
\[
[d\omega_\bC]_{134}=[d\omega_\bC]_{14\bar4}=0, \qquad
[d\omega_\bC]_{13\bar2}=[d\omega_\bC]_{23\bar2}=0\,, 
\]
we conclude $B=P=0$. Moreover, $[d\omega_\bC]_{23\bar1}=\bar N\,\theta=0$ make us distinguish two different cases:
\begin{itemize}
\item Suppose $N=0$. Then, we observe that 
$d\varphi^4\in\Lambda^2\langle \varphi^1,\varphi^2,\varphi^{\bar1},\varphi^{\bar 2} \rangle$.
Defining
\begin{equation}\label{cambio}
\tau^1=\varphi^1,\quad \tau^2=\varphi^2,\quad \tau^3=\varphi^4,\quad \tau^4=\varphi^3,
\end{equation}
one finds a basis for $\mfg^{*(1,0)}$ for which Theorem~\ref{nilpotent-J} holds.
\item Let $N\neq0$, thus $\theta=0$. Bearing in mind that also $\zeta=0$, \eqref{no-deg} is simply $\gamma\tau\neq 0$. Moreover, as $B=P=0$, the Jacobi identity
$d^2\varphi^4=0$ gives the equations
\[
N\bar K=0, \qquad N\bar H=0.
\]
Consequently, $H=K=0$. As we are interested in the SnN case, recall that one needs
$(B,E,H)\neq (0,0,0)$. Since $B=H=0$, necessarily $E\neq 0$. However, from
$[d\omega_\bC]_{12\bar4}=E\tau=0$ one gets $\tau=0$, but this contradicts the non-degeneracy of $\omega_\bC$.
\end{itemize}
This concludes the proof for those $(\mfg,J)$ parametrized by Theorem~\ref{SnN-8-d}~\textrm{(i)}.

Let us now consider an $8$-dimensional NLA~$\mfg$ endowed with a complex structure~$J$ whose
complex structure equations follow Theorem~\ref{SnN-8-d}~\textrm{(ii)}. We first
observe that $E\neq 0$ to avoid being in the conditions of Theorem~\ref{nilpotent-J}. A direct computation shows
\begin{align*}
 d\omega_\bC &=(-D\,\zeta - E\,\tau)\,\varphi^{124}
	+ (L\,\theta + F\,\tau + \bar M\,\gamma)\,\varphi^{12\bar1}
	 - (i\,s\,\gamma - M\,\theta - D\,\tau)\,\varphi^{12\bar2}\\
	 &+N\,\theta\,\varphi^{12\bar3} + E\,\tau\,\varphi^{12\bar4}
	+ (L\,\zeta - A\,\tau + \bar N\,\gamma)\,\varphi^{13\bar1} 
	+ M\,\zeta\,\varphi^{13\bar2} + N\,\zeta\,\varphi^{13\bar3}\\
	&-(F\,\zeta + A\,\theta)\,\varphi^{14\bar1} 
	- D\,\zeta\,\varphi^{14\bar2} - E\,\zeta\,\varphi^{14\bar4}  
	- (\bar M\,\zeta - \bar N\,\theta)\,\varphi^{23\bar1}
	+ i\,s\,\zeta\,\varphi^{23\bar2}\,.
\end{align*}

From $[d\omega_\bC]_{12\bar4}=[d\omega_\bC]_{14\bar4}=0$, we obtain $\tau=\zeta=0$ and
the non-degeneracy condition \eqref{no-deg} becomes $\beta\theta\neq 0$. As a consequence,
$[d\omega_\bC]_{13\bar1}=0$ yields $N=0$. Defining a new basis as in \eqref{cambio}, 
one sees that Theorem~\ref{nilpotent-J} is satisfied, and thus $J$ is nilpotent.

We finally consider the case of a pair $(\mfg,J)$ parametrized by Theorem~\ref{SnN-8-d}~\textrm{(iii)}.
First, we observe that $B=E=0$ implies that $J$ is nilpotent (due to Theorem~\ref{nilpotent-J}), 
so we assume
\[
 (B,E)\neq (0,0)\,.
\]
Now, we compute
\begin{align*}
 d\omega_\bC &= - E\,\tau\,\varphi^{124}
	+ (L\,\theta + F\,\tau + \bar M\,\gamma)\,\varphi^{12\bar1}
	- (i\,s\,\gamma - M\,\theta)\,\varphi^{12\bar2} \\
& - (P\,\gamma - N\,\theta)\,\varphi^{12\bar 3} + E\,\tau\,\varphi^{12\bar4} + B\,\tau\,\varphi^{134} 
	+ (L\,\zeta - A\,\tau + \bar N\,\gamma)\,\varphi^{13\bar1} \\
& + (M\,\zeta + \bar P\,\gamma)\,\varphi^{13\bar2} - (i\,t\,\gamma - N\,\zeta)\,\varphi^{13\bar3} - B\,\tau\,\varphi^{13\bar 4} 
	- (F\,\zeta + A\,\theta)\,\varphi^{14\bar1} \\
& - (E\,\zeta + B\,\theta)\,\varphi^{14\bar4}  - (\bar M\,\zeta - \bar N\,\theta)\,\varphi^{23\bar1}
	+ (i\,s\,\zeta + \bar P\,\theta)\,\varphi^{23\bar2}\\
&	+ (P\,\zeta - i\,t\,\theta)\,\varphi^{23\bar3}\,.
\end{align*}

Since $[d\omega_\bC]_{124}=[d\omega_\bC]_{134}=0$, one needs $\tau=0$.
We next distinguish two cases depending on the value of the parameter $B$.

If we let $B=0$, then $E\neq 0$ and $[d\omega_\bC]_{14\bar4}=0$ gives $\zeta=0$.
Bearing in mind that also $\tau=0$, the non degeneration condition \eqref{no-deg} 
simplifies to $\beta\theta\neq 0$. Then, $[d\omega_\bC]_{14\bar1}=[d\omega_\bC]_{23\bar1}=
[d\omega_\bC]_{23\bar2}=[d\omega_\bC]_{23\bar3}=0$ gives
$A=N=P=t=0$, so applying \eqref{cambio} one finds a basis for $\mfg^{*(1,0)}$ in terms
of which Theorem~\ref{nilpotent-J} holds.

If $B\neq 0$, define a new basis for $\mfg^{*(1,0)}$ as follows
\[\begin{cases}
\ \text{for } E=0: \quad
\tau^1=\varphi^1,\quad \tau^2=\varphi^3,\quad \tau^3=\varphi^2,\quad \tau^4=\varphi^4,\\[5pt]
\ \text{for } E\neq0: \quad
\tau^1=\varphi^1,\quad \tau^2=\varphi^2-\frac BE\,\varphi^3,\quad \tau^3=\varphi^3,\quad \tau^4=\varphi^4\,.
\end{cases}
\]
In both cases, one checks that the new structure equations follow 
Theorem~\ref{SnN-8-d} \textrm{(iii)} with $B=0$. We have already seen that the
existence of a complex symplectic form on such $(\mfg,J)$ implies that $J$ is nilpotent. This concludes our proof.
\end{proof}

\begin{corollary}\label{co:redox8d}
Every complex symplectic nilpotent Lie algebra of dimension $8$ is the complex symplectic oxidation of a nilpotent complex symplectic Lie algebra of dimension $4$.
\end{corollary}
\begin{proof}
By Proposition \ref{hs-implica-nilp}, the complex structure of an 8-dimensional complex symplectic nilpotent Lie algebra $(\mfg,J,\omega)$ is quasi-nilpotent, hence we can apply Corollary \ref{co:redox_nilpotent}
\end{proof}

\begin{remark}
The proof of Proposition \ref{hs-implica-nilp} relies on the above mentioned classification of SnN complex structures on 8-dimensional nilpotent Lie algebras. Two natural questions arise:
\begin{itemize}
\item Is it possible to prove the result without appealing to the classification?
\item Are there examples of complex symplectic structures $(J,\omega)$ on nilpotent Lie algebras of dimension $4n$, $n\geq 3$, such that $J$ is SnN?
\end{itemize}
\end{remark}

\subsection{Complex symplectic oxidation data in dimension four}
Our next goal is to explicitly obtain all complex symplectic nilpotent Lie algebras in dimension 8. To do this, we need to determine all possible nilpotent complex symplectic oxidation data on the 4-dimensional complex symplectic nilpotent Lie algebras $(\bR^4,J_0,\omega_0)$ and $(\mfh_3\oplus\bR,J_0,\omega_0)$.\\

We start with nilpotent complex symplectic oxidation data on $(\mfh_3\oplus\bR,J_0,\omega_0)$
and use here and in the next proposition the notations from Lemma~\ref{le:oxidationdata}:

\begin{proposition}\label{pro:oxidationdatah3+R}
Let $(f,S,\tau)$ be nilpotent complex symplectic oxidation data on $(\mfh_3\oplus\bR,J_0,\omega_0)$. Then there exists a non-zero vector $v\in V$ and real numbers $\alpha_i,\beta_i,\gamma_i\in\bR$, $i=1,2,4$, such that with the notations of Lemma \ref{le:oxidationdata} we have
\[
 S_{11}=\alpha_1e^1+\alpha_2e^2+\alpha_4e^4, \ S_{22}=\beta_1e^1+\beta_2e^2-\alpha_4e^4 \ \mathrm{and} \ 
S_{12}=\gamma_1e^1+\gamma_2e^2+\gamma_4 e^4
\]
and $f_1=\left(\begin{smallmatrix}0 & 0\\ A & 0\end{smallmatrix}\right)$, $f_2=\left(\begin{smallmatrix} 0 & 0\\ B & 0 \end{smallmatrix}\right)$, where the $2\times 2$-matrices $A,B$ have one of the following forms:
\begin{itemize}
  \item[(i)] $A=\begin{pmatrix} 2 & 0\\ 0 & 0 \end{pmatrix}$, $B=\begin{pmatrix} 0 & -2\\ 0 & 0 \end{pmatrix}$ and $\alpha_4\neq 0$ or $\gamma_4\neq 0$;
  \item[(ii)] $A=\begin{pmatrix} a+1 & -b\\ b & a-1 \end{pmatrix}$, $B=\begin{pmatrix} c & -d-1\\ d-1 & c \end{pmatrix}$ with $a,b,c,d\in\bR$ and $\alpha_4=\gamma_4=0$;
  \item[(iii)] $A=\begin{pmatrix} 1 & 0\\ 0 & 1 \end{pmatrix}$, $B=\begin{pmatrix} a & -b\\ b & a \end{pmatrix}$ with $b\neq 0$ and $\alpha_4=\gamma_4=0$;
  \item[(iv)] $A=\begin{pmatrix}\cos(\phi) & -\sin(\phi)\\ \sin(\phi) & \cos(\phi)\end{pmatrix}$, $B=0$, $\phi\in\left[0,2\pi\right)$ and $\alpha_4=\gamma_4=0$;
  \item[(v)] $A=B=0$.
\end{itemize}
\end{proposition}
\begin{proof}
Let $(f,S,\tau)$ be nilpotent complex symplectic oxidation data, choose some non-zero vector $v\in V$ and set $v_1\coloneqq v$ and $v_2\coloneqq Iv$. Then
$f_1\coloneqq f(v_1), f_2\coloneqq f(v_2)$ are nilpotent derivations. We first investigate how such a nilpotent derivation $\varphi\in\mathrm{Der}(\mfh_3\oplus\bR)$ looks like in general. Thereto, observe that $\left[\mfh_3\oplus\bR,\mfh_3\oplus\bR\right]=\la e_3\ra $ and $\mathfrak{z}(\mfh_3\oplus\bR)=\la e_3,e_4\ra $. As $\varphi$ has to preserve
these spaces and is additionally nilpotent, we must have $\varphi(e_3)=0$ and $\varphi(e_4)\in\la e_3\ra$. Now $\varphi([X,Y])=[\varphi(X),Y]+[X,\varphi(Y)]$ is automatically fulfilled if $X$ or $Y$ is central and we only need to check this equation for $X=e_1$, $Y=e_2$, which gives us $0=\varphi([e_1,e_2])=[\varphi(e_1),e_2]+[e_1,\varphi(e_2)]$.

We denote $f_1(e_k)=\sum_{i=1}^4x_{ik}\,e_i$ and
$f_2(e_k)=\sum_{i=1}^4y_{ik}\,e_i$. Then, in matrix form with respect to the basis $\{e_1,\ldots,e_4\}$, one has
\begin{equation*}
f_1=\begin{pmatrix} x_{11} & x_{12} & 0 & 0\\ x_{21} & -x_{11} & 0 & 0\\ x_{31} & x_{32} & 0 & x_{34}\\ x_{41} & x_{42} & 0 & 0 \end{pmatrix}, \quad f_2=\begin{pmatrix} y_{11} & y_{12} & 0 & 0\\ y_{21} & -y_{11} & 0 & 0\\ y_{31} & y_{32} & 0 & y_{34}\\ y_{41} & y_{42} & 0 & 0 \end{pmatrix}
\end{equation*}
with $x_{ij}, y_{ij}\in\bR$ such that $x_{11}^2+x_{12}\, x_{21}=0$ and $y_{11}^2+y_{12}\, y_{21}=0$, the latter two conditions coming from the fact that $f_1$
and $f_2$ have to be nilpotent.

Let us now impose the condition $h\coloneqq f_2-J_0\circ f_1\in\mathfrak{sp}(\mfh_3\oplus \bR,\omega_0, J_0)$, as required in Lemma~\ref{le:oxidationdata}. First observe that $h(e_3)=0$ and $h(e_4)=y_{34}e_3+x_{34}e_4$. But $h$ is an element of $\mathfrak{sp}(\mfh_3\oplus \bR,\omega_0, J_0)$, so we must have
$h(e_4)=0$ as well, i.e. $x_{34}=y_{34}=0$. Moreover, the condition $h.\omega_0=0$ gives us $\omega_0(h(e_i),e_j)=0$ for all $i=1,2$ and all $j=3,4$, which implies that
$h(e_1),h(e_2)\in \la e_3,e_4\ra^{\perp}=\la e_3,e_4\ra$. 
Performing the appropriate computations, one checks that the latter is equivalent to $-y_{21}=x_{11}=-y_{12}$, $x_{21}=y_{11}=x_{12}$, which together with
$x_{11}^2+x_{12}\, x_{21}=0=y_{11}^2+y_{12}\, y_{21}$ implies $x_{11}=x_{12}=x_{21}=y_{11}=y_{12}=y_{21}=0$. 
This gives us
$f_1=\left(\begin{smallmatrix} 0 & 0 \\ A & 0 \end{smallmatrix}\right)$ and $f_2=\left(\begin{smallmatrix} 0 & 0 \\ B & 0 \end{smallmatrix}\right)$,
for some $2\times 2$ matrices
\[
A\coloneqq \begin{pmatrix} x_{31} & x_{32} \\ x_{41} & x_{42} \end{pmatrix} \quad \textrm{and} \quad B\coloneqq \begin{pmatrix} y_{31} & y_{32} \\ y_{41} & y_{42} \end{pmatrix}.
\]
Hence,
$h.\omega_0=0$ is fulfilled. We now need to impose $[h,J_0]=0$.
Setting $J\coloneqq \left( \begin{smallmatrix} 0 & 1 \\ -1 & 0 \end{smallmatrix}\right)$, the previous condition is equivalent to $[B-JA,J]=0$.
Decomposing $A=A_1+A_2$ and $B=B_1+B_2$ into their $J$-commutative parts $A_1$, $B_1$ and $J$-anti-commutative parts $A_2$ and $B_2$, the condition becomes $B_2=J A_2$.

Let us first assume that $A_2\neq 0$. Then $A_2=\left(\begin{smallmatrix} \tilde{a} & \tilde{b} \\ \tilde{b} & -\tilde{a} \end{smallmatrix}\right)$
and so $B_2=\left(\begin{smallmatrix} \tilde{b} & -\tilde{a} \\ -\tilde{a} & -\tilde{b} \end{smallmatrix}\right)$
for certain $(\tilde{a},\tilde{b})\in \bR^2\setminus \{(0,0)\}$. Since $A_2$ and $B_2$ are linearly independent, there is a linear combination of them which is equal to $\left(\begin{smallmatrix} 1 & 0 \\ 0 & -1\end{smallmatrix}\right)$. Hence, by replacing the non-zero vector $v$ by a linear combination of $v$ and $Jv$, we may assume that $A_2=\left(\begin{smallmatrix} 1 & 0 \\ 0 & -1\end{smallmatrix}\right)$
and consequently $B_2=\left(\begin{smallmatrix} 0 & -1 \\ -1 & 0 \end{smallmatrix}\right)$. Then, we have $A=\left(\begin{smallmatrix} a+1 & -b \\ b & a-1 \end{smallmatrix}\right)$ and $B=\left(\begin{smallmatrix} c & -d-1 \\ d-1 & c \end{smallmatrix}\right)$ for certain $a,b,c,d\in \bR$ in this case.

We now suppose $A_2=0$, which implies $B_2=0$ as well. Thus, $A=\left(\begin{smallmatrix} \tilde{a} & -\tilde{b} \\ \tilde{b} & \tilde{a} \end{smallmatrix}\right)$ and $B=\left(\begin{smallmatrix} a & -b \\ b & a \end{smallmatrix}\right)$ for certain $\tilde{a},\tilde{b},a,b\in \bR$. If $(\tilde{a},\tilde{b})$ and $(a,b)$ are linearly independent, we can replace $v$ as before by a linear combination of $v$ and $Jv$ such that we may assume $\tilde{a}=1$, $\tilde{b}=0$ and $b\neq 0$. If they are linearly dependent but not both zero, then we may assume that $a=b=0$ and normalize $\tilde{a}^2+\tilde{b}^2=1$, i.e. $\tilde{a}=\cos(\phi)$, $\tilde{b}=\sin(\phi)$ for some $\phi\in [0,2\pi)$. Finally, the last case is $A=0$ and $B=0$.

Summarizing we have the following possibilities for $A$ and $B$:
\begin{itemize}
  \item $A=\left(\begin{smallmatrix} a+1 & -b\\ b & a-1 \end{smallmatrix}\right)$, $B=\left(\begin{smallmatrix} c & -d-1 \\ d-1 & c \end{smallmatrix}\right)$, $a,b,c,d\in\bR$,
  \item $A=\left(\begin{smallmatrix} 1 & 0\\ 0 & 1 \end{smallmatrix}\right)$, $B=\left(\begin{smallmatrix} a & -b \\ b & a \end{smallmatrix}\right)$ with $a,b\in\bR$ and $b\neq 0$,
  \item $A=\left(\begin{smallmatrix} \cos(\phi) & -\sin(\phi)\\ \sin(\phi) & \cos(\phi) \end{smallmatrix}\right)$, $B=0$, $\phi\in \left[0,2\pi \right)$,
  \item $A=0$, $B=0$.
\end{itemize}

Now, we investigate the remaining data $S_{11}$, $S_{22}$ and $S_{12}$. By Lemma~\ref{le:oxidationdata}~(ii), we need $[f_1,f_2]=[J_0(S_{11}^{\sharp}+S_{22}^{\sharp}),\cdot]_{\bar{\mfg}}$, where $\bar\mfg=\mfh_3\oplus\bR$. In our situation $f_1$ and $f_2$ commute. Hence, $J_0(S_{11}^{\sharp}+S_{22}^{\sharp})$ lies in $\mfz(\mfh_3\oplus\bR)=\la e_3,e_4\ra$. Writing $S_{11}=\sum_{i=1}^4\alpha_ie^i$, $S_{22}=\sum_{i=1}^4\beta_ie^i$, $S_{12}=\sum_{i=1}^4\gamma_ie^i$ for $\alpha_i,\beta_i,\gamma_i\in\bR$, we obtain
\[
J_0(S_{11}^{\sharp}+S_{22}^{\sharp}) = (\alpha_3+\beta_3)e_1-(\alpha_4+\beta_4)e_2-(\alpha_1+\beta_1)e_3+(\alpha_2+\beta_2)e_4\,,
\]
and so $\beta_3=-\alpha_3$, $\beta_4=-\alpha_4$. Hence, $S_{11}+S_{22}\in\la e_1,e_2\ra$.

Now look at the conditions (iii) in Lemma \ref{le:oxidationdata}. Since $f_i\circ f_j=0$ for $i,j=1,2$ and $\mathrm{im}(f_1),\mathrm{im}(f_2)\subseteq \la e_3,e_4\ra$, we have $f_i.f_j.\omega_0=0$.
Moreover, $J_0^*(S_{11}+S_{22})\in \la e_1,e_2\ra$ implies $d_{\bar\mfg}(J_0^*(S_{11}+S_{22}))=0$. Hence, the conditions (iii) in Lemma \ref{le:oxidationdata} are equivalent to $S_{11}$, $S_{22}$ and $S_{12}$ 
all being closed, i.e. to $\alpha_3=\beta_3=\gamma_3=0$.

Finally, consider the conditions (iv) in Lemma \ref{le:oxidationdata}. Since $J_0(S_{11}^{\sharp}+S_{22}^{\sharp})\in\la e_3,e_4\ra\subset\ker f_1,\ker f_2$, the first summand on the right hand side of the conditions vanishes. Moreover, $\mathrm{im}(J_0\circ f_i) \subset\la e_3,e_4\ra$ for $i=1,2$ and $S_{11}+S_{22}$ vanishes on $\la e_3,e_4\ra$, so the initial expressions (iv) are equivalent to
\begin{equation}\label{aux}
S_{12}\circ f_1 - S_{11}\circ f_2=0, \quad S_{22}\circ f_1-S_{12}\circ f_2=0\,. 
\end{equation}

We evaluate both equations individually for the above possibilities for~$A$ and~$B$. Thereto, notice that both equations are automatically fulfilled if evaluated on $\la e_3,e_4\ra\subset\ker f_1,\ker f_2$. Hence, we only have to check the equations on $\la e_1,e_2\ra$:
\begin{itemize}
\item
If $A=\left(\begin{smallmatrix} a+1 & -b\\ b & a-1 \end{smallmatrix}\right)$, $B=\left(\begin{smallmatrix} c & -d-1 \\ d-1 & c \end{smallmatrix}\right)$, $a,b,c,d\in\bR$, then we obtain by inserting $e_1$ and $e_2$ the equations
\begin{align*}
  \gamma_4b-\alpha_4(d-1) &=0, & \gamma_4(a-1)-\alpha_4c &=0,\\ -\alpha_4 b-\gamma_4(d-1) &= 0, &  -\alpha_4(a-1)-\gamma_4c &=0.
\end{align*}
These equations are equivalent to
\begin{equation*}
Lx:=\begin{pmatrix} 0 & \gamma_4 & 0 & -\alpha_4\\ \gamma_4 & 0 & -\alpha_4 & 0\\ 0 & -\alpha_4 & 0 & -\gamma_4\\ -\alpha_4 & 0 & -\gamma_4 & 0 \end{pmatrix} \begin{pmatrix} a-1\\ b\\ c\\ d-1 \end{pmatrix} = \begin{pmatrix} 0\\ 0\\ 0\\ 0 \end{pmatrix}
\end{equation*}
One computes $\det(L)=(\alpha_4^2+\gamma_4^2)^2$. Therefore, if $\alpha_4\neq 0$ or $\gamma_4\neq 0$, then $a=1$, $b=0$, $c=0$ and $d=1$. This is case (i). If $\alpha_4=\gamma_4=0$, then there are no further restrictions for $a,b,c,d$. This is case (ii).
\item
If $A=\left(\begin{smallmatrix} 1 & 0\\ 0 & 1 \end{smallmatrix}\right)$, $B=\left(\begin{smallmatrix} a & -b \\ b & a \end{smallmatrix}\right)$, the equations \eqref{aux} are equivalent to $\alpha_4=\gamma_4=0$, giving us case (iii).
\item
If $A=\left(\begin{smallmatrix} c & -s\\ s & c \end{smallmatrix}\right)$, $B=0$ with $c=\cos(\phi)$, $s=\sin(\phi)$ for $\phi\in \left[0,2\pi \right)$, the equations \eqref{aux} are equivalent
to $\alpha_4=\gamma_4=0$. This is case (iv).
\item
If $A=0$ and $B=0$ then there are no further restrictions to $\alpha_4$ and $\gamma_4$ and we obtained case (v).
\end{itemize}
\end{proof}

We move to the nilpotent complex symplectic oxidation data on $(\bR^4,J_0,\omega_0)$. Thereto, we first prove two lemmas which will turn out to be useful for that computation:
\begin{lemma}\label{le:oxidationdata2}
Let $(\mfg,J,\omega)$ be a 4-dimensional complex symplectic Lie algebra and let $W\subseteq \mfg$ be a two-dimensional $J$-invariant subspace. Then there is an adapted basis $\{X_1,X_2,X_3,X_4\}$ for $(\mfg,J,\omega)$ such that $\{X_3,X_4\}$ is a basis of $W$.
\end{lemma}
\begin{proof}
Choose an arbitrary adapted basis $\{Y_1,Y_2,Y_3,Y_4\}$ for $(\mfg,J,\omega)$. Obviously, $\mathrm{SL}(2,\bC)$ acts transitively on $\bC^2\setminus \{(0,0)\}$ and so $\mathrm{SL}(2,\bC)$, considered as a subgroup of $\GL(4,\bR)$, acts transitively on complex one-dimensional subspaces of $\bR^4$. As $\{Y_1,Y_2,Y_3,Y_4\}$ identifies
$\mfg$ with $\bR^4$ as vector spaces such that $J$-invariant subspaces of $\mfg$ correspond to complex-one-dimensional subspaces of $\bR^4$, there is an element $A\in \mathrm{SL}(2,\bC)\subseteq \GL(4,\bR)$ such that for $\{X_1,X_2,X_3,X_4\}:=\{Y_1,Y_2,Y_3,Y_4\}\cdot A$, 
the vectors $X_3$ and $X_4$ span $W$. Now by Remark \ref{re:adaptedbasis}, $\{X_1,X_2,X_3,X_4\}$ is an adapted basis for $(\mfg,J,\omega)$ as well and the claim follows.
\end{proof}

\begin{lemma}\label{le:oxidationdata1}
Let $(f,S,\tau)$ be a nilpotent complex symplectic oxidation datum on $(\bR^4,J_0,\omega_0)$. Then there exist a two-dimensional $J_0$-invariant subspace $W$
such that $f(v)|_W=0$ for all $v\in V$.
\end{lemma}
\begin{proof}
Let $v\in V\setminus \{0\}$ and set $f_1$, $f_2$. We need to show that $f_1|_W=f_2|_W=0$. Thereto, we note that, as $\bR^4$ is Abelian, all one-forms are closed, so condition~(ii) in Lemma~\ref{le:oxidationdata} simplifies to $\left[f_1,f_2\right]=0$ and
condition~(iii) becomes $f_i.f_j.\omega_0=0$, for all $i,j=1,2$.

Since $f_1$ and $f_2$ are nilpotent and commute, there is a common eigenvector $X_1$ with eigenvalue $0$ of $f_1$ and $f_2$.
Moreover, Lemma~\ref{le:oxidationdata}~(i) states that $h\coloneqq f_2-J_0\circ f_1$ commutes with $J_0$ and so we have $h(X_2)=0$ for $X_2\coloneqq -J_0 X_1$ as well, i.e.~$f_2(X_2)=J_0(f_1(X_2))$.

Note that by Lemma~\ref{le:oxidationdata2} one can complete $\{X_1,X_2\}$ up to an adapted basis for $(\mfg_0,J_0,\omega_0)$. Hence, for any $Y\in \bR^4$ and any $i\in \{1,2\}$ the identity
$0=(h.\omega_0)(Y,X_i)=\omega_0(h(Y),X_i)$ holds,
which implies $\mathrm{im}(h)\subseteq \la X_1,X_2\ra^{\perp}=\la X_1,X_2\ra$.

First assume that $f_1(X_2)\in \la X_1,X_2\ra$. If $f_1(X_2)\neq 0$, then there is a linear combination
$\tilde{f}$ of $f_1$ and $f_2$ with $\tilde{f}(X_2)=X_2$, so $\tilde{f}$ is 
not nilpotent, which is a contradiction. Hence, we must have $f_1(X_2)=0$ and so $f_2(X_2)=0$ as well, i.e.~$W\coloneqq \la X_1,X_2\ra$ is a
subspace of $\bR^4$ as desired.

Next, assume that $X_3\coloneqq f_1(X_2)\notin \la X_1,X_2\ra$. Then, $f_2(X_2)=J_0 X_3\coloneqq -X_4$. 

On the one hand, we compute
\begin{align*}
0 &= (f_1.f_1.\omega_0)(X_1,X_2) = \omega_0\big(X_1,f_1(X_3)\big),\\
0 &= (f_1.f_2.\omega_0)(X_1,X_2) = \omega_0\big(X_1,f_2(X_3)\big)  \\
	& = \omega_0\big(X_1,h(X_3) + J_0\circ f_1(X_3)\big) 
	= \omega_0\big(J_0X_1, f_1(X_3)\big) = -\omega_0\big(X_2,f_1(X_3)\big),
\end{align*}
and similarly one also has 
\begin{align*}
0 &= (f_2.f_1.\omega_0)(X_1,X_2) = -\omega_0\big(X_1,f_1(X_4)\big),\\
0 &= (f_2.f_2.\omega_0)(X_1,X_2) = -\omega_0\big(X_1,f_2(X_4)\big) = \omega_0\big(X_2,f_1(X_4)\big).
\end{align*}
Hence, we can conclude $f_1(X_3),\, f_1(X_4)\in \langle X_1, X_2\rangle^{\perp}=
\langle X_1, X_2\rangle$. Moreover, observe that
\begin{align*}
\omega_0\big(X_2, f_2(X_3)\big) &= \omega_0\big(X_2, h(X_3) + J_0\circ f_1(X_3)\big) 
	= \omega_0\big(J_0X_2,f_1(X_3)\big)\\ 
	& =\omega_0\big(X_1,f_1(X_3)\big)=0,\\
\omega_0\big(X_2, f_2(X_4)\big) &= \omega_0\big(X_2, h(X_4) + J_0\circ f_1(X_4)\big) 
	= \omega_0\big(J_0X_2,f_1(X_4)\big)\\ &=\omega_0\big(X_1,f_1(X_4)\big)=0,\\
\end{align*}
so also $f_2(X_3),\, f_2(X_4)\in \langle X_1, X_2\rangle$.

On the other hand, since $(f_1\circ f_1)(X_2)=f_1(X_3)\in\langle X_1, X_2\rangle$ and 
$f_1$ is a nilpotent derivation,
one can ensure $f_1(X_3)\in\langle X_1\rangle$. Similarly, $f_2(X_4)\in\langle X_1\rangle$.
Now, as $f_1(X_1)=f_2(X_1)=0$ and $f_1$, $f_2$ commute, one concludes that $(f_i\circ f_j)(X_k)=0$, for
every $i,j=1,2$ and $k=3,4$. Consequently,
\begin{equation*}
\begin{split}
0&=(f_1.f_1.\omega_0)(X_2,X_3)=\omega_0(f_1(X_3),X_3)+2\omega_0(X_3,f_1(X_3))=-\omega_0(f_1(X_3),X_3),\\
0&=(f_1.f_2.\omega_0)(X_2,X_3)=-\omega_0(X_4,f_1(X_3)),\\
0&=(f_2.f_1.\omega_0)(X_2,X_4)=\omega_0(X_3,f_2(X_4)),\\
0&=(f_2.f_2.\omega_0)(X_2,X_4)=-\omega_0(f_2(X_4),X_4)-2\omega_0(X_4,f_2(X_4))=\omega_0(f_2(X_4),X_4).
\end{split}
\end{equation*}
This implies $f_1(X_3),\,f_2(X_4)\in\la X_3,X_4\ra^{\perp}=\la X_3,X_4\ra$.
Moreover, $f_1(X_4) =-f_1(f_2(X_2))=-f_2(f_1(X_2))=-f_2(X_3)$, thus
\begin{equation*}
\begin{split}
0&=(f_1.f_1.\omega_0)(X_2,X_4)=2\omega_0(X_3,f_1(X_4))=-2\omega_0(X_3,f_2(X_3)),\\
0&=(f_2.f_2.\omega_0)(X_2,X_3)=-2\omega_0(X_4,f_2(X_3))=2\omega_0(X_4,f_1(X_4)).
\end{split}
\end{equation*}
Hence, $f_1(X_4)=f_2(X_3)\in\la X_3,X_4\ra$ as well.

As a consequence, $f_1(X_3),\,f_1(X_4),\,f_2(X_3),\,f_2(X_ 4)\in 
\la X_1,X_2\ra \cap \la X_3,X_4\ra = \{0\}$, and we can conclude that
$W\coloneqq \la X_3,X_4\ra$ is a subspace of $\bR^4$ as desired.
\end{proof}

Now we are able to compute the possible complex symplectic oxidation data on $(\bR^4,J_0,\omega_0)$:
\begin{proposition}\label{pro:oxidationdataR4}
Let $(f,S,\tau)$ be a nilpotent complex symplectic oxidation datum on $(\bR^4,J_0,\omega_0)$. Then there exists a non-zero vector $v\in V$ and real numbers $\alpha_i$, $\beta_i$, $\gamma_i\in\bR$, $i=1,2,3,4$, such that, up to an isomorphism of complex symplectic oxidation data, we have
\[
S_{11} = \sum_{i=1}^4\alpha_ie^i, \qquad S_{22} = \sum_{i=1}^4\beta_ie^i, \qquad S_{12} = \sum_{i=1}^4\gamma_ie^i\,, 
\]
 and $f_1=f(v)=\left(\begin{smallmatrix}0 & 0\\ A & 0\end{smallmatrix}\right)$, $f_2=f(Iv)=\left(\begin{smallmatrix} 0 & 0\\ B & 0 \end{smallmatrix}\right)$, where the $2\times 2$-matrices $A,B$ have one of the following form:
\begin{itemize}
	\item[(i)] $A=\left(\begin{smallmatrix} a+1 & 0\\ 0 & a-1 \end{smallmatrix}\right)$, $B=\left(\begin{smallmatrix} b & -c-1\\ c-1 & b \end{smallmatrix}\right)$ with $a\geq 0$ and
	$(\alpha_3,\alpha_4,\beta_3,\beta_4,\gamma_3,\gamma_4)^T\in \bR^6$ lying in the kernel of
\begin{equation*}
L:=\left(\begin{smallmatrix} -b & \tfrac{a-2c+7}{2} & 0 & \tfrac{a+5}{2} & a+1 & 0 \\
                                \tfrac{2c-a+7}{2} & -b & \tfrac{5-a}{2} & 0 & 0 & a-1 \\
                                \tfrac{1-c}{2} & \tfrac{b}{2} & \tfrac{2a-c+3}{2} & \tfrac{b}{2} & -b & 1-c \\
                                -\tfrac{b}{2} & -\tfrac{c+1}{2} & -\tfrac{b}{2} & \tfrac{2a-c-3}{2} & \ c+1 & -b
        \end{smallmatrix}\right)\,;
\end{equation*}
  \item[(ii)] $A=\begin{pmatrix} 1 & 0\\ 0 & 1 \end{pmatrix}$, $B=\begin{pmatrix} a & -b\\ b & a \end{pmatrix}$ with $b\neq 0$ and $(a,b)\neq (0,1)$ and
	\begin{align*}
	  \beta_3 = b\,\alpha_3-a\,\alpha_4, \ \beta_4 = a\,\alpha_3+b\,\alpha_4, \ 
\gamma_3= \tfrac{a}{2}\,\alpha_3+\tfrac{b-1}{2}\,\alpha_4, \ \gamma_4= \tfrac{1-b}{2}\,\alpha_3+\tfrac{a}{2}\,\alpha_4\,;
	\end{align*}
	\item[(iii)] $A=\begin{pmatrix} 1 & 0\\ 0 & 1 \end{pmatrix}$, $B=\begin{pmatrix} 0 & -1\\ 1 & 0 \end{pmatrix}$ and $\gamma_3=\frac{1}{2}(\alpha_4-\beta_4)$, $\gamma_4=-\frac{1}{2}(\alpha_3-\beta_3)$;
	\item[(iv)] $A=\begin{pmatrix} 1 & 0\\ 0 & 1 \end{pmatrix}$, $B=0$ and $\beta_3=\beta_4=0$ and $\gamma_3=-\tfrac{1}{2}\,\alpha_4$, $\gamma_4=\tfrac{1}{2}\,\alpha_3$ or
	\item[(v)] $A=0$, $B=0$.
\end{itemize}
\end{proposition}

\begin{proof}
Let $(f,S,\tau)$ be oxidation data on $\bR^4$ and $v\in V\setminus \{0\}$. Set $f_1:=f(v)$, $f_2:=f(Iv)$ and $h:=f_2-J_0\circ f_1$. By Lemma \ref{le:oxidationdata1}, we have a two-dimensional $J_0$-invariant subspace $W$ such that $f_1|_W=f_2|_W=0$ and by Lemma \ref{le:oxidationdata2}, there is an adapted basis $\{X_1,X_2,X_3,X_4\}$ for $(\bR^4,J_0,\omega_0)$ such that $X_3,X_4$ span $W$.

Now note that we have $h(X_3)=h(X_4)=0$ and so $\mathrm{im}(h)\subseteq 
\la X_3,X_4\ra^{\perp}= \la X_3,X_4\ra$ as for any $Y\in \bR^4$ and any $j\in \{3,4\}$ the identity $0=(h.\omega_0)(Y,X_j)=\omega_0(h(Y),X_j)$ is true. Consequently, if we consider the projection $\pi\colon\bR^4\to \la X_1,X_2\ra$ along $\la X_3,X_4\ra $, then we have $\pi\big( f_2(X_i) \big) = J_0\big( \pi\big(f_1(X_i)\big) \big)$, for any
$i=1,2$.
Hence, these projections have to be zero as otherwise we may arrange a linear combination
$\tilde{f}$ of $f_1$ and $f_2$ such that $\pi\big( \tilde{f}(X_i) \big)=X_i$ for some $i\in \{1,2\}$, a contradiction
to $\tilde{f}$ being nilpotent.

Thus, $f_1=\left(\begin{smallmatrix} 0 & 0 \\ A & 0 \end{smallmatrix}\right)$ and $f_2=\left(\begin{smallmatrix} 0 & 0 \\ B & 0 \end{smallmatrix}\right)$ for certain $A,B\in \bR^{2\times 2}$
with respect to the adapted basis $\{X_1,\ldots,X_4\}$. As in the proof of Proposition \ref{pro:oxidationdatah3+R}, we decompose $A=A_1+A_2$ and $B=B_1+B_2$ into
their $J$-commutative parts $A_1$, $B_1$ and $J$-anti-commutative
parts $A_2$ and $B_2$, being $J\coloneqq\left( \begin{smallmatrix} 0 & 1 \\ -1 & 0 \end{smallmatrix}\right)$. Then,
the condition $h\in \mathfrak{sp}(\bR^4,J_0,\omega_0)$ reduces to $B_2=J A_2$.
Arguing as in the proof of Proposition \ref{pro:oxidationdatah3+R},
we get the same possible forms for $A$ and $B$. Now note that rotating in the $\la X_1,X_2\ra$-plane by an angle of 
$\tfrac{\phi}{2}$ and in the $\la X_3,X_4\ra$-plane by an angle of $-\tfrac{\phi}{2}$ is an automorphism
of $(\bR^4,J_0,\omega_0)$. This automorphism does not change the $J$-anti-commutative parts $A_2$ and $B_2$ but rotates the $J$-commutative parts by an angle of $\phi$. Then:

In the first case, i.e. $A=\left(\begin{smallmatrix} a+1 & -b\\ b & a-1 \end{smallmatrix}\right)$,
$B=\left(\begin{smallmatrix} c & -d-1 \\ d-1 & c \end{smallmatrix}\right)$ for certain $a,b,c,d\in\bR$, we may apply this automorphism to be able to assume that $a\geq 0$ and $b=0$.

In the third case, i.e. $A=\left(\begin{smallmatrix} \cos(\phi) & -\sin(\phi)\\ \sin(\phi) & \cos(\phi) \end{smallmatrix}\right)$, $B=0$ for some $\phi\in \left[0,2\pi \right)$, we may apply this automorphism to be able to assume that $A=I_2$ and $B=0$. 

Hence, we have the following possibilities for $A$ and $B$:

\begin{itemize}
  \item $A=\left(\begin{smallmatrix} a+1 & 0 \\ 0 & a-1 \end{smallmatrix}\right)$, $B=\left(\begin{smallmatrix} b & -c-1 \\ c-1 & b \end{smallmatrix}\right)$ for certain $a,b,c\in\bR$ with $a\geq 0$,
  \item $A=\left(\begin{smallmatrix} 1 & 0\\ 0 & 1 \end{smallmatrix}\right)$, $B=\left(\begin{smallmatrix} a & -b \\ b & a \end{smallmatrix}\right)$ for certain $a,b\in\bR$ with $b\neq 0$,
  \item $A=\left(\begin{smallmatrix} 1 & 0\\ 0 & 1 \end{smallmatrix}\right)$, $B=0$ or
  \item $A=0$, $B=0$.
\end{itemize}
Now note that since $\bR^4$ is Abelian, conditions (i) -- (iii) in Lemma~\ref{le:oxidationdata} are fulfilled for all these $A$s and $B$s and we only have to look at condition 
(iv). Thereto, we write $S_{11}=\sum_{i=1}^4\alpha_ie^i$, $S_{22}=\sum_{i=1}^4\beta_ie^i$, $S_{12}=\sum_{i=1}^4\gamma_i e^i$ for $\alpha_i,\beta_i,\gamma_i\in\bR$
and obtain
\begin{equation*}
J_0(S_{11}^{\sharp}+S_{22}^{\sharp}) = (\alpha_3+\beta_3)e_1-(\alpha_4+\beta_4)e_2-(\alpha_1+\beta_1)e_3+(\alpha_2+\beta_2)e_4
\end{equation*}
as in the proof of Proposition \ref{pro:oxidationdatah3+R}. Note that condition (iv) in Lemma \ref{le:oxidationdata} is automatically fulfilled in the last case, giving us part $(v)$ in Proposition \ref{pro:oxidationdataR4}.  Thus, we only have to discuss the first three cases:

First assume that $A=\left(\begin{smallmatrix} a+1 & 0 \\ 0 & a-1 \end{smallmatrix}\right)$, $B=\left(\begin{smallmatrix} b & -c-1 \\ c-1 & b \end{smallmatrix}\right)$ for certain $a,b,c\in\bR$ with $a\geq 0$. Then condition $(iv)$ in Lemma \ref{le:oxidationdata} gives us
\begin{equation*}
\begin{split}
(-c+1)\alpha_4+(a+1)\gamma_3-b\alpha_3&=-\tfrac{a+5}{2}(\alpha_4+\beta_4),\\
(c+1)\alpha_3+(a-1)\gamma_4-b\alpha_4&=\tfrac{a-5}{2}(\alpha_3+\beta_3),\\
(-c+1)\gamma_4+(a+1)\beta_3-b\gamma_3&=\tfrac{c-1}{2}(\alpha_3+\beta_3)-\tfrac{b}{2}(\alpha_4+\beta_4),\\
(c+1)\gamma_3+(a-1)\beta_4-b\gamma_4&=\tfrac{c+1}{2}(\alpha_4+\beta_4)+\tfrac{b}{2}(\alpha_3+\beta_3).
\end{split}
\end{equation*}
This is equivalent to $(\alpha_3,\alpha_4,\beta_3,\beta_4,\gamma_3,\gamma_4)^T\in \bR^6$ lying in the kernel of
\begin{equation*}
L=\left(\begin{smallmatrix} -b & -c+1+\tfrac{a+5}{2} & 0 & \tfrac{a+5}{2} & a+1 & 0 \\
                                c+1-\tfrac{a-5}{2} & -b & -\tfrac{a-5}{2} & 0 & 0 & a-1 \\
                                -\tfrac{c-1}{2} & \tfrac{b}{2} & a+1-\tfrac{c-1}{2} & \tfrac{b}{2} & -b & -c+1 \\
                                -\tfrac{b}{2} & -\tfrac{c+1}{2} & -\tfrac{b}{2} & a-1-\tfrac{c+1}{2} & c+1 & -b
        \end{smallmatrix}\right)\,,
\end{equation*}
obtaining case (i) in the statement above.

Now we consider the case $A=\left(\begin{smallmatrix} 1 & 0\\ 0 & 1 \end{smallmatrix}\right)$, $B=\left(\begin{smallmatrix} a & -b\\ b & a \end{smallmatrix}\right)$ with $b\neq 0$. Then, by
condition (iv) in Lemma \ref{le:oxidationdata}, we have
\begin{equation*}
\begin{split}
  \gamma_3-a\alpha_3-b\alpha_4 &= -\tfrac{1}{2}(\alpha_4+\beta_4),\\
	\gamma_4+b\alpha_3-a\alpha_4 &= \tfrac{1}{2}(\alpha_3+\beta_3),\\
	\beta_3-a\gamma_3-b\gamma_4 &= \tfrac{b}{2}(\alpha_3+\beta_3)-\tfrac{a}{2}(\alpha_4+\beta_4),\\
	\beta_4+b\gamma_3-a\gamma_4 &= \tfrac{a}{2}(\alpha_3+\beta_3)+\tfrac{b}{2}(\alpha_4+\beta_4).
\end{split}
\end{equation*}
The first two equations imply
\begin{equation*}
\begin{split}
  \gamma_3 &= a\alpha_3+(b-\tfrac{1}{2})\alpha_4-\tfrac{1}{2}\beta_4,\\
	\gamma_4 &= (\tfrac{1}{2}-b)\alpha_3+a\alpha_4+\tfrac{1}{2}\beta_3.
\end{split}
\end{equation*}
Inserting into the last two equations gives
\begin{equation*}
\begin{split}
  (b^2-a^2-b)\alpha_3+a(1-2b)\alpha_4+(1-b)\beta_3+a\beta_4 &= 0,\\
	a(2b-1)\alpha_3+(b^2-a^2-b)\alpha_4-a\beta_3+(1-b)\beta_4 &= 0.
\end{split}
\end{equation*}
If $(a,b)\neq (0,1)$, then the above equations imply
\begin{equation*}
\begin{split}
\beta_3 &= b\alpha_3-a\alpha_4,\\
\beta_4 &= a\alpha_3+b\alpha_4.
\end{split}
\end{equation*}
This is case (ii). If $(a,b)=(0,1)$, then the equations give case (iii).

Finally, we have to consider the case $A=\left(\begin{smallmatrix} 1 & 0\\ 0 & 1 \end{smallmatrix}\right)$ and $B=0$. Then, the computations coincide with
those of the previous case, so it suffices to take $a=b=0$ in the above equations, which imply $\beta_3=\beta_4=0$ and $\gamma_3=-\tfrac{1}{2}\alpha_4$, $\gamma_4=\tfrac{1}{2}\alpha_3$. This is case (iv).
\end{proof}

\begin{corollary}\label{cor:nilpstep}
The nilpotency step of an 8-dimensional complex symplectic nilpotent Lie algebra is $\leq 4$
and all nilpotency steps between $1$ and $4$ may occur.
\end{corollary}
\begin{proof}
By Corollary \ref{co:redox8d}, any 8-dimensional complex symplectic nilpotent Lie algebra $(\mfg,J,\omega)$
is the complex symplectic oxidation a four-dimensional complex symplectic Lie algebra. Hence, it is obtainable by the complex symplectic 
oxidation data on $(\mfh_3\oplus \bR,J_0,\omega_0)$ given in Proposition \ref{pro:oxidationdatah3+R} or by the complex symplectic 
oxidation data on $(\bR^4,J_0,\omega_0)$ given in Proposition \ref{pro:oxidationdataR4}. In the following, we use the notation for 
complex symplectic oxidation from Section \ref{oxidation-reduction} and note that we have $\mfg=V\oplus \bar\mfg\oplus V^*$ as vector 
space with $\bar\mfg\in\{\mfh_3\oplus \bR,\bR^4\}$. Surely, we have $V^*\subseteq \mfg_1$. Then we see from 
Proposition \ref{pro:oxidationdatah3+R} and Proposition \ref{pro:oxidationdataR4} that in both cases there is a two-dimensional 
subspace $U$ such that $\im(f(v))\subseteq U\subseteq \ker(f(v))$ for all $v\in V$ and such that $U$ is $\bar\mfg$-central and contains
the commutator ideal of $\bar\mfg$. From the form of the Lie bracket of $\mfg$ given in
\eqref{eq:Liebracketong1} -- \eqref{eq:Liebracketong3}, we deduce that $[U,\mfg]_{\mfg}\subseteq V^*\subseteq \mfg_1$ and so $U\oplus V^*\subseteq \mfg_2$. But then the same equations imply $[\bar\mfg,\mfg]_{\mfg}\subseteq V^*\oplus U\subseteq \mfg_2$,
i.e. $\bar\mfg\oplus V^*\subseteq \mfg_3$. Finally, the mentioned equations show that $[\mfg,\mfg]\subseteq \bar\mfg\oplus V^*\subseteq \mfg_3$ and so $\mfg=\mfg_4$, i.e. the nilpotency step of $\mfg$ is at most~$4$.

Now for the existence of $8$-dimensional nilpotent complex symplectic Lie algebras with nilpotency step $1$ -- $4$ note that $\bR^8$ and $\mfh_3\oplus \bR^5$ provide examples of nilpotency step $1$ and $2$, respectively, obtainable from $(\bR^4,J_0,\omega_0)$ and $(\mfh_3\oplus \bR,J_0,\omega_0)$, respectively, with trivial complex symplectic oxidation data.

A nilpotency step $3$ example may be obtained by choosing $\tau=0$, $S_{11}=e^3$, $S_{12}=e^1$ and $S_{22}=0$ in case (v) in Proposition \ref{pro:oxidationdataR4}. Then the non-zero Lie brackets (up to anti-symmetry) on the vector space $\mfg=V\oplus \bR^4\oplus V^*$ with respect to the basis $\{v_1,v_2,e_1,\ldots,e_4,v^1,v^2\}$ are given by
\begin{equation*}
[v_1,v_2]=e_1,[v_1,e_1]=v^2, [v_1,e_3]=v^1, [v_1,e_4]=\tfrac{1}{2}v^2, [v_2,e_1]=v^1, [v_2,e_4]=-\tfrac{1}{2}v^1
\end{equation*}
and one has $\mfg_1=\mathfrak{z}(\mfg)=V^*$, $\mfg_2=\bar\mfg\oplus V^*$ and $\mfg_3=\mfg$.

Finally, a nilpotency step $4$ example may be obtained by choosing $\tau=0$, $S_{11}=e^3$, $S_{12}=e^1$ and $S_{22}=\tfrac{1}{2}e^4$ in case (iv) in Proposition \ref{pro:oxidationdataR4}. In this case, the non-zero Lie brackets (up to anti-symmetry) on the vector space $\mfg=V\oplus \bR^4\oplus V^*$ with respect to the basis $\{v_1,v_2,e_1,\ldots,e_4,v^1,v^2\}$ are given by
\begin{equation*}
\begin{split}
[v_1,v_2]&=e_1-\tfrac{1}{2} e_2,\quad [v_1,e_1]=e_3+v^2,\quad [v_1,e_2]=e_4, \quad [v_1,e_3]=v^1-\tfrac{1}{4} v^2,\\
[v_1,e_4]&=\tfrac{1}{2}v^2,\quad [v_2,e_1]=v^1,\quad [v_2,e_3]=\tfrac{1}{4}v^1,\quad [v_2,e_4]=-\tfrac{1}{2}v^1+\tfrac{1}{2}v^2.
\end{split}
\end{equation*}
Here, we have $\mfg_1=\mathfrak{z}(\mfg)=V^*$, $\mfg_2=\la e_3,e_4\ra\oplus V^*$, $\mfg_3=\bar\mfg\oplus V^*$
and $\mfg_4=\mfg$.
\end{proof}

\begin{remark}\label{rem:nilpstep}
It was shown in \cite[Theorem 2.1]{Dotti-Fino} that an 8-dimensional hypercomplex nilpotent Lie algebra is necessarily 2-step nilpotent. Compared to this, the existence of complex symplectic structures on nilpotent Lie algebras seems to be less restrictive. On the other hand, the 8-dimensional Lie algebra $(0,0,12,13,14,15,16,17)$ has nilpotency step 7 and carries the symplectic structure $\omega=e^{18}+e^{27}-e^{36}+e^{45}$. Moreover, the 8-dimensional Lie algebra obtained by setting $A=E=F=H=K=M=P=s=0$, $B=-1$ and $C=D=G=L=N=1$ in Theorem \ref{SnN-8-d} (i) has ascending type $(1,3,5,6,8)$ and nilpotency step 5; its center is 1-dimensional, hence the complex structure is strongly non-nilpotent and the Lie algebra admits no complex symplectic structure.
\end{remark}

\section{Complex symplectic structures on nilmanifolds and solvmanifolds}\label{examples}

In this section we apply the previous constructions to provide several examples of compact manifolds with complex symplectic structures. Our examples will be nilmanifolds and solvmanifolds; as we recalled in the introduction, a nilmanifold (resp. solvmanifold) is the quotient of a connected, simply connected nilpotent (resp. solvable) Lie group $G$ by a lattice~$\Gamma$ (i.e.~a discrete and cocompact subgroup). Every such $G$ is homeomorphic to $\bR^n$ for some $n$, and the natural projection $G\to \Gamma\backslash G$ exhibits $G$ as the universal cover of $\Gamma\backslash G$; thus $\pi_1(\Gamma\backslash G)=\Gamma$ and nilmanifolds (resp. solvmanifolds) are aspherical spaces. 

By a result of Mal'tsev \cite{Malcev}, a connected, simply connected, nilpotent Lie group $G$ contains a lattice if and only if the associated Lie algebra $\mfg$ has a rational structure, that is, a rational subalgebra $\mfg_\bQ\subset \mfg$ such that $\mfg=\mfg_\bQ\otimes\bR$. Nothing is known, in general, about lattices in solvable Lie groups.

If $G$ is connected, simply connected and nilpotent, Nomizu's theorem \cite{Nomizu} says that the de Rham cohomology of the nilmanifold $\Gamma\backslash G$ is isomorphic to the Lie algebra cohomology of $\mfg$, $H^\bullet_{\textrm{dR}}(\Gamma\backslash G)\cong H^\bullet(\mfg^*)$. Hattori proved an analogous result for \emph{completely solvable} solvmanifolds, see \cite{Hattori}.

As a consequence of Nomizu's theorem, if a nilmanifold has a symplectic structure it also has a left-invariant one. By Hattori's result, the same happens for completely solvable solvmanifolds. Concerning complex symplectic structures, we have the following refinement:

\begin{proposition}\label{prop:invariance}
Let $N=\Gamma\backslash G$ be a solvmanifold endowed with a left-invariant complex structure $J$
such that $H^\bullet_{\textrm{dR}}(\Gamma\backslash G)\cong H^\bullet(\mfg^*)$.
If $(N,J)$ admits a complex symplectic structure $\tilde\omega$ (not necessarily left-invariant), then there exists a left-invariant complex symplectic structure $(J,\omega)$ on $N$.
\end{proposition}

To prove the proposition, we need to recall a few facts. Let $N$ be a $2n$-dimensional manifold. An almost complex structure $J$ on $N$ acts on the space of real $2$-forms $\Omega^2(N)$ as an involution:
\[
\begin{array}{rrcl}
    J^*\colon  & \Omega^2(N) & \longrightarrow & \Omega^2(N). \\[4pt]
	& \omega(\cdot,\cdot) & \longrightarrow& \omega(J\cdot,J\cdot)
\end{array} 
\]

Therefore, one has the splitting
\[
\Omega^2(N)=\Omega^+(N) \oplus \Omega^-(N)\,, 
\]
where
$\Omega^\pm (N)=\{\omega\in\Omega^2(N) \mid J^*\omega=\pm \omega \}$.

Denote by $\mathcal Z_J^+(N)$ (resp.~$\mathcal Z_J^-(N))$ the space of closed $2$-forms that are $J$-invariant (resp. $J$-anti-invariant). The following subspaces were introduced in \cite{Li-Zhang}, in relation with Donaldson's ``tamed to compatible'' conjecture:

\[
H_J^{\pm}(N):=\{[ \alpha ] \in H_{\textrm{dR}}^2(N) \mid \alpha\in \mathcal Z_J^\pm(N)\}\subseteq H_{\textrm{dR}}^2(N)\,. 
\]

Suppose $(J,\omega)$ is a complex symplectic structure on a manifold $N$; then $\omega(J\cdot, \cdot)=\omega(\cdot, J\cdot)$ or, equivalently, $\omega(J\cdot, J\cdot)=-\omega(\cdot, \cdot)$, so $\omega\in\Omega^-(N)$. Since $d\omega$ is closed and non-degenerate, $0\neq [\omega]\in H_J^-(N)$.


\begin{proof}[Proof of the proposition]
Let $\mathfrak g$ be the (real) Lie algebra of $G$. As a consequence of \cite[Theorem 5.4]{Angella-Tomassini-Zhang}, the condition $H^\bullet_{\textrm{dR}}(\Gamma\backslash G)\cong H^\bullet(\mfg^*)$ implies that $H_J^-(N)\cong H_J^-(\mathfrak g)$. Therefore, one can find a left-invariant $2$-form $\omega\in\Lambda^2\mfg^*$ such that $[\omega]=[\tilde\omega]\in H_J^-(N)$. Observe that $\omega$ is anti-$J$-invariant and satisfies $d\omega=0$. To prove the result, it suffices to check that $\omega^n\neq 0$, where $2n=\dim N$. 
Notice that
$$0\neq [\tilde\omega]^n=[\omega]^n=[\omega^n],$$
where $\omega^n\in\Lambda^{2n}\mfg^*$ is a top degree form. Consequently, if $\{e^1,\dots,e^{2n}\}$ is a basis for $\mfg^*$,
then $\omega^n=k\,e^{1}\wedge\dots\wedge e^{2n}$, for some $k\in\mathbb R$. Due to the previous equality between cohomology classes, 
$k\neq 0$ and thus $\omega^n\neq 0$, as desired.
\end{proof}

Combining Proposition~\ref{hs-implica-nilp} with the previous result, we can conclude the following:

\begin{corollary}
Let $N$ be an $8$-dimensional nilmanifold endowed with an invariant complex structure $J$.
If the Lie algebra underlying $N$ has $1$-dimensional center, then $(N,J)$ cannot admit
a complex symplectic structure.
\end{corollary}

\begin{example}\label{ex:csandhc}

Consider the 8-dimensional 2-step nilpotent Lie algebra
\[
 \mfg=\mfq\mfh_7\oplus\bR=(0,0,0,0,0,12-34,13+24,14-23)\,,
\]
which is the direct sum of the quaternionic Heisenberg algebra $\mfq\mfh_7$ with $\bR$.
$\mfg$ is precisely the Lie algebra denoted by $\mathfrak n_3$ in \cite{Dotti-Fino}. 
Consider the complex structures $I$, $J$ and $K$ whose spaces of $(1,0)$-forms are, respectively,
\begin{table}[ht!]
\begin{center}
\begin{tabular}{cccc}
{\tabulinesep=1.2mm
\begin{tabu}{ccc}
  $\varphi_I^1$\!\! & $=$\!\! & $e^1+ie^2$,\\
  $\varphi_J^1$\!\! & $=$\!\! & $e^1+ie^3$,\\
  $\varphi_K^1$\!\! & $=$\!\! & $e^1+ie^4$,
 \end{tabu}}
 &\!\!\!\!
 {\tabulinesep=1.2mm
 \begin{tabu}{ccc}
  $\varphi_I^2$\!\! & $=$\!\! & $e^3+ie^4$,\\
  $\varphi_J^2$\!\! & $=$\!\! & $e^2-ie^4$,\\
  $\varphi_K^2$\!\! & $=$\!\! & $e^2+ie^3$,
 \end{tabu}}
 &\!\!\!\!
 {\tabulinesep=1.2mm
 \begin{tabu}{ccc}
  $\varphi_I^3$\!\! & $=$\!\! & $e^5+ie^6$,\\
  $\varphi_J^3$\!\! & $=$\!\! & $e^5+ie^7$,\\
  $\varphi_K^3$\!\! & $=$\!\! & $e^5+ie^8$,
 \end{tabu}}
 &\!\!\!\!
 {\tabulinesep=1.2mm
 \begin{tabu}{ccc}
  $\varphi_I^4$\!\! & $=$\!\! & $e^7+ie^8$,\\
  $\varphi_J^4$\!\! & $=$\!\! & $e^6-ie^8$,\\
  $\varphi_K^4$\!\! & $=$\!\! & $e^6+ie^7$.
 \end{tabu}}
\end{tabular}
\end{center}
\end{table}

According to \cite[Page 54]{Dotti-Fino}, $\{I,J,K\}$ is a hypercomplex structure. Moreover, one checks easily that $\omega=\frac{1}{2}e^{18}+\frac{1}{2}e^{27}+e^{36}+e^{45}$ is a symplectic form and $(I,\omega)$ is a complex symplectic structure. $\mfg$ has an obvious rational structure, hence $G$, the unique connected, simply connected nilpotent Lie group with Lie algebra $\mfg$, contains a lattice $\Gamma$, and $N=\Gamma\backslash G$ is a compact nilmanifold with a hypercomplex and a complex symplectic structure, both invariant. Using Nomizu's theorem, one sees that $b_1(N)=5$, hence $N$ is not diffeomorphic to the torus $T^8$, which, by a result of Hasegawa \cite{Hasegawa}, is a necessary condition for a nilmanifold to admit a K\"ahler structure. Thus, $N$ has also no hyperk\"ahler structure.

Consider the $I$-invariant, central ideal $\mfa\subset\mfg$ with basis $\{e_6,e_5\}$; notice that $\mfa\cong V^*$, where $V$ has basis $\{-e_3,-e_4\}$. Then $\mfa^\perp/\mfa$ is isomorphic to
$\bR^4$ with basis $\left\{\sqrt{2}e_2,\sqrt{2}e_1,\sqrt{2}e_8,\sqrt{2}e_7\right\}$. The complex symplectic Lie algebra $(\mfg,I,\omega)$ can be obtained from Proposition~\ref{pro:oxidationdataR4} case (i)
with $a=b=c=0$, $\tau(e_3,e_4)=e_6$ and $S=0$.

\end{example}

\begin{example}\label{ex:hcnocomplsymp}
Consider the 8-dimensional 2-step nilpotent Lie algebra
\[
\mfg=(0,0,0,0,0,0,0,12-34)\,;
\]
notice that $\mfg$ is isomorphic to $\mfh_5\oplus\bR^3$, where $\mfh_{2n+1}$ is the Heisenberg Lie algebra, with basis $\{X_1,Y_1,\ldots,X_n,Y_n,Z\}$ and non-zero brackets $[X_i,Y_i]=Z$ for $i=1,\ldots,n$. It is easy to see that $\mfh_{2n+1}\oplus\bR^{2m+1}$ is not symplectic for $n\geq 2$, so in particular, it is not complex symplectic (for $n=1$, recall that $\mfh_3\oplus\mathbb R^5$ admits a complex symplectic structure, as remarked in the proof of Corollary~\ref{cor:nilpstep}). Notice that $\mfg$ is obtained by setting all parameters equal to zero in \cite[Page 54]{Dotti-Fino}, except for $d_1=1$, hence the complex structures $I$, $J$ and $K$ described in the previous example give a hypercomplex structure $\{I,J,K\}$ on $\mfg$.
Since $\mfg$ has a rational structure, $G$, the unique connected, simply connected nilpotent Lie group with Lie algebra $\mfg$ contains a lattice $\Gamma$, and $N=\Gamma\backslash G$ is a compact nilmanifold with a left-invariant hypercomplex structure. By Nomizu's theorem, $N$ has no symplectic structure, hence, in particular, no complex symplectic structure.

\end{example}

\begin{example}
 Consider the family of Lie algebras $\mfg=\mfg(A,B,C)$ with complex structure $J$ given in terms of a basis $\{\varphi^1,\ldots,\varphi^4\}$ of $\mfg^{*(1,0)}$ with structure equations
\[\begin{cases}
d\varphi^1=d\varphi^2=0,\\
d\varphi^3=\varphi^{1\bar2},\\
d\varphi^4= A\,\varphi^{13} + B\,\varphi^{1\bar1} + C\,\varphi^{2\bar 2},
\end{cases}
\]
where $A,B,C\in\mathbb C$ are such that $AC\neq 0$. The complex 2-form
\begin{equation}\label{csympform}
\omega_\bC=\alpha\,\varphi^{12}+\beta\,\varphi^{13}+\gamma\,\varphi^{14}+C\gamma\,\varphi^{23}, \quad \alpha,\beta,\gamma\in\mathbb C, \quad\gamma\neq 0 
\end{equation}
is in $\Lambda^2\mfg^{*(1,0)}$ and is closed and non-degenerate, hence $(J,\omega_\bC)$ is a complex symplectic structure on $\mfg$. Defining the following real basis of $\mfg$
\[
\begin{cases}
e^1 = \frac 12\, (\varphi^1 + \varphi^{\bar 1})\\
e^2 = \frac 12\, (\varphi^2 + \varphi^{\bar 2})\\
e^3 = -\frac i2\, (\varphi^1 - \varphi^{\bar 1})\\
e^4 = -\frac i2\, (\varphi^2 - \varphi^{\bar 2})
\end{cases}
\qquad
\begin{cases}
e^5 = \frac 12\, (\varphi^3 + \varphi^{\bar 3})\\
e^6 = \frac i2\, (\varphi^3 - \varphi^{\bar 3})\\
e^7 = \frac 12\, (\varphi^4 + \varphi^{\bar 4})\\
e^8 = -\frac i2\, (\varphi^4 - \varphi^{\bar 4})
\end{cases}
\]
we get the real structure equations of $\mfg$:
\[
\begin{cases}
de^1 = de^2 = de^3 = de^4 = 0,\\
de^5 = e^{12} + e^{34},\\
de^6 = e^{14} + e^{23},\\
de^7 = 2\,\im\,B\,e^{13} + \re\,A\,e^{15} + \im\,A\,e^{16} 
	+ 2\,\im\,C\,e^{24} - \im\,A\,e^{35} + \re\,A\,e^{36},\\
de^8 = -2\,\re\,B\,e^{13} + \im\,A\,e^{15} - \re\,A\,e^{16} 
	- 2\,\re\,C\,e^{24} + \re\,A\,e^{35} + \im\,A\,e^{36}.\\
\end{cases}
\]
Notice that $A\neq 0$ ensures that $\mfg$ has nilpotency step 3, while $C\neq 0$ is needed for the non-degeneration of the complex symplectic form $\omega_\bC$. As we already pointed out in Remark \ref{rem:nilpstep}, $\mfg$ carries no hypercomplex structure by a result of Dotti and Fino. Fixing for instance $A=C=1$, $B=0$, we see that $\mfg$ has a rational structure, hence the corresponding connected, simply connected, nilpotent Lie group $G$ admits a lattice $\Gamma$, and $N=\Gamma\backslash G$ carries a complex symplectic structure but no left-invariant hypercomplex structures.

The center of $\mfg$ is $\mfz(\mfg)=\la e_7,e_8\ra$ which is $J$-invariant. One sees that $\mfa^{\perp}/\mfa$ is isomorphic to $\bR^4=\la e_4,e_2,e_5,e_6\ra$. Furthermore, if one sets $\alpha=0$, $\beta=0$ and $\gamma=1$ in \eqref{csympform}, the complex symplectic Lie algebra $(\mfg,J,\Re(\omega_\bC))$ can be obtained by complex symplectic oxidation from Proposition \ref{pro:oxidationdataR4} case (i) by setting $a=b=c=0$, $S_{11}=-e^3$, $S_{12}=-e^4$, $S_{22}=e^3$ and $\tau=0$, with $v_1=-e_3$, $v_2=e_1$, $v^1=-e_8$ and $v^2=-e_7$.

\end{example}

\begin{example}
Consider the complex Heisenberg group
\[
H_3^\bC=\left\{\begin{pmatrix} 1 & z_1 & z_3\\ 0 & 1 & z_2\\ 0 & 0 & 1\end{pmatrix} \mid z_1,z_2,z_3\in\bC\right\}\,;
\]
notice that it is a nilpotent Lie group. $\Xi=\{A\in H_3^\bC \mid z_1,z_2,z_3\in \bZ[i]\}\subset H_3^\bC$ is a lattice and $\Xi\bs H_3^\bC$ is the \emph{Iwasawa manifold}. We set $G=H_3^\bC\times\bC$ and $\Gamma=\Xi\times\bZ[i]$. Thus $\Gamma\bs G$ is a nilmanifold. A basis of left-invariant holomorphic 1-forms on $G$ is given by
\[
\varphi^1=dz_1, \quad \varphi^2=dz_2, \quad \varphi^3=dz_3-z_1dz_2 \quad \textrm{and} \quad \varphi^4=dz_4\,,
\]
$z_4$ being the coordinate on the $\bC$ factor. Then
\[
d\varphi^1=0, \quad d\varphi^2=0, \quad d\varphi^3=-\varphi^{12}\quad \textrm{and} \quad d\varphi^4=0\,.
\]
The Lie algebra of $G$ is $\mfg=\mfh_3^\bC\oplus\bC=(\mfh_3\oplus\bR)\otimes\bC$, the complexification of $\mfh_3\oplus\bR$. As shown in \cite[Example 4.1]{Cattaneo-Tomassini}, $\Gamma\bs G$ admits the following family of complex symplectic structures:
\begin{equation}\label{sform_CT}
\omega_\bC=\alpha\varphi^{12}+\beta\varphi^{13}+\gamma\varphi^{14}+\delta\varphi^{23}+\varepsilon\varphi^{24}\,.
\end{equation}
This example was also considered by Guan in \cite{Guan1}.
We now show how, for selected values of these parameters, $(\mfh_3^\bC\oplus\bC,J,\omega_\bC)$ can be obtained by oxidation (notice that here the complex structure $J$ is given in terms of the basis $\{\varphi^1,\ldots,\varphi^4\}$ of $(1,0)$-forms). We fix the following real basis of 1-forms
\[
\begin{cases}
\varphi^1 = e^1-ie^2\\
\varphi^2 = e^3-ie^4\\
\varphi^3 = e^7-ie^8\\
\varphi^4 = e^5-ie^6
\end{cases}
\]
and let $\{e_i\}_{i=1}^8$ be the dual basis. Thus $\mfg=(0,0,0,0,0,0,-13+24,-14-23)$. 

Choosing $\omega_\bC=\varphi^{13}-\varphi^{24}$ in \eqref{sform_CT}, and taking its real part, we see that $\omega=e^{17}-e^{28}-e^{35}+e^{46}$ is a symplectic form with respect to which $J$ is symmetric. The central ideal $\mfa=\langle e_6,-e_5\rangle\cong V^*$ is clearly $J$-invariant and $\mfa^\perp/\mfa\cong\bR^4=\la -e_2,e_1,e_7,e_8\ra$; choosing $\{-e_4,-e_3\}$ as a basis of $V$, the complement of $\mfa^\perp$ in $\mfg$, the corresponding oxidation data are $\tau=0$, $S=0$ and $A,B$ as in case (ii) of Proposition \ref{pro:oxidationdataR4} with $a=0$ and $b=-1$.


Picking $\omega_\bC=\varphi^{14}+\varphi^{23}$ in \eqref{sform_CT}, and taking the real part, we obtain the symplectic form $\omega=e^{15}-e^{26}+e^{37}-e^{48}$, with respect to which $J$ is symmetric. $\mfa=\langle e_7,e_8\rangle\cong V^*$ is a central, $J$-invariant ideal and $\mfa^\perp/\mfa\cong\bR^4=\la-e_2,e_1,e_5,e_6\ra$. In the notation of Lemma \ref{le:oxidationdata}, we have $v_1=-e_3$ and $v_2=e_4$, spanning $V$; moreover, $S_{11}=e^1$, $S_{12}=e^2$ and $S_{22}=-e^1$. The other oxidation data are $\tau=0$, $f=0$, i.e. we are in case (v) of Proposition \ref{pro:oxidationdataR4}.

\end{example}

\begin{example}
We show that our construction also works in the more general context of solvable Lie algebras. Consider the compact complex manifold $X_0$ which is the product of the Nakamura threefold and a complex torus. $X_0$ has the structure of a solvmanifold $\Gamma\bs G$, where $G$ is a certain solvable complex Lie group and $\Gamma\subset G$ is a lattice. A global, left-invariant frame of holomorphic 1-forms on $G$ is given by
\[
\varphi^1=dz_1, \quad \varphi^2=e^{-z_1}dz_2, \quad \varphi^3=e^{z_1}dz_3 \quad \textrm{and} \quad \varphi^4=dz_4\,.
\]
In terms of  $\{\varphi^1,\ldots,\varphi^4\}$, the structure equations of $\mfg$ are
\[
d\varphi^1=0, \quad d\varphi^2=-\varphi^{12}, \quad d\varphi^3=\varphi^{13} \quad \textrm{and} \quad d\varphi^4=0\,.
\]

The form $\omega_{\bC}=\alpha\varphi^{14}+\beta\varphi^{23}$, $\alpha\beta\neq 0$, is holomorphic, closed and non-degenerate, hence defines a complex symplectic structure on $X_0$. $X_0$ does not satisfy the $\partial\bar{\partial}$-lemma, but some deformations of its complex structures $X_t$, $t\in\bC-\{0\}$, do, see \cite[Section 4]{Angella-Kasuya}. In fact, $X_0$ and its deformations (all of which admit complex symplectic structures) have been considered in \cite[Example 4.2]{Cattaneo-Tomassini}. We fix the following real basis of 1-forms
\[
\begin{cases}
\varphi^1 = e^1-ie^2\\
\varphi^2 = e^3-ie^4\\
\varphi^3 = e^5-ie^6\\
\varphi^4 = e^7-ie^8
\end{cases}
\]
and let $\{e_i\}_{i=1}^8$ be the dual basis. One sees that 
\[
\mfg=(0,0,-13+24,-14-23,15-26,16+25,0,0)\,.
\]

We choose $\omega_\bC=\varphi^{14}+\varphi^{23}$; its real part $\omega=e^{17}-e^{28}+e^{35}-e^{46}$ is a symplectic form with respect to which $J$ is symmetric. The central ideal $\mfa=\langle e_7,e_8\rangle\cong V^*$ is clearly $J$-invariant and $\mfa^\perp/\mfa\cong\bR^4=\la e_3,e_4,-e_6,e_5\ra$; choosing $\{-e_1,e_2\}$ as a basis of $V$, the complement of $\mfa^\perp$ in $\mfg$, the corresponding oxidation data are $\tau=0$, $S=0$ and the derivations
\[
f(-e_1)=\begin{pmatrix} -1 & 0 & 0 & 0\\ 0 & -1 & 0 & 0\\ 0 & 0 & 1 & 0\\ 0 & 0 & 0 & 1\end{pmatrix},\quad f(e_2)=\begin{pmatrix} 0 & -1 & 0 & 0\\ 1 & 0 & 0 & 0\\ 0 & 0 & 0 & 1\\ 0 & 0 & -1 & 0\end{pmatrix}\,.
\]
with respect to the adapted basis $\{e_3,e_4,-e_6,e_5\}$ of $(\bR^4,J_0,\omega_0)$.


\end{example}

\medskip\noindent
{{\bf Acknowledgements.} The authors would like to thank Ilka Agricola, Daniele Angella, S\"onke Rollenske and Luis Ugarte for useful conversations. The first author was supported by a \emph{Juan de la Cierva - Incorporaci\'on} Fellowship of Spanish Ministerio de Ciencia, Innovaci\'on y Universidades. The second author was partly supported by a \emph{Forschungsstipendium} (FR 3473/2-1) from the Deutsche Forschungsgemeinschaft (DFG). The third author was partially supported by the project MTM2017-85649-P (AEI/FEDER, UE) and the group E22-17R ``Algebra y Geometr\'{\i}a'' (Gobierno de Arag\'on/FEDER).


\end{document}